\documentclass[12pt,x11names]{article}
\usepackage{cite}
\usepackage{geometry,amsmath,amssymb,amsthm,graphicx,epsfig}
\usepackage{subfig,enumerate,dcolumn,latexsym}
\usepackage{fancybox}
\usepackage[graph]{xy}
\usepackage[colorlinks=true,linkcolor=blue,citecolor=blue]{hyperref}
\usepackage{color,epstopdf}
\usepackage[usenames,dvipsnames,svgnames,table]{xcolor}
\numberwithin{equation}{section}

\newtheorem{thm}{Theorem}[section]

\newtheorem{cor}[thm]{Corollary}
\newtheorem{propn}[thm]{Proposition}
\newtheorem{defn}[thm]{Definition}
\newtheorem{assumption}[thm]{Assumption}

\newtheorem{algorithm}{Algorithm}[section]
\newtheorem{eg}[thm]{Example}
\newtheorem{remark}[thm]{Remark}

\newcommand{\paren}[1]{\left(#1\right)}

\newcommand{\klam}[1]{\left\{#1\right\}}

\newcommand{\Ccal}{{\mathcal C}}

\newcommand{\Jcal}{{\mathcal J}}

\newcommand{\Ocal}{{\mathcal O}}
\newcommand{\Pcal}{{\mathcal P}}

\renewcommand{\Bbb}{\mathbb{B}}

\newcommand{\Nbb}{\mathbb{N}}

\newcommand{\Rbb}{\mathbb{R}}

\newcommand{\xhat}{{\widehat{x}}}

\newcommand{\bbar}{{\overline{b}}}

\newcommand{\pbar}{{\overline{p}}}

\newcommand{\ubar}{{\overline{u}}}
\newcommand{\vbar}{{\overline{v}}}

\newcommand{\xbar}{{\overline{x}}}
\newcommand{\ybar}{{\overline{y}}}
\newcommand{\zbar}{{\overline{z}}}

\newcommand{\alphabar}{{\overline{\alpha}}}

\newcommand{\epsilonbar}{{\overline{\epsilon}}}

\newcommand{\rhobar}{{\overline{\rho}}}

\newcommand{\Ltilde}{{\widetilde{L}}}

\renewcommand{\equiv}{:=}

\newcommand{\toinf}{\rightarrow\infty}

\newcommand{\Rn}{{\Rbb^n}}

\newcommand{\Rm}{{\Rbb^m}}

\newcommand{\Rp}{{\Rbb_+}}
\newcommand{\Rth}{{\Rbb^3}}
\newcommand{\Rtw}{{\Rbb^2}}
\newcommand{\extre}{(-\infty,+\infty]}

\newcommand{\ucmin}[2]{\underset{#2}{\mbox{minimize}}~#1}

\newcommand{\cmin}[3]{\begin{array}{ll}\underset{#2}{\mbox{minimize }}\qquad #1&\\
                       \mbox{subject to }\qquad#3&\end{array}}

\newcommand{\set}[2]{\left\{#1\,\left|\,#2\right.\right\}}
\newcommand{\map}[3]{#1:\,#2\rightarrow #3\,}
\newcommand{\mmap}[3]{#1:\,#2\rightrightarrows #3\,}

\newcommand{\ip}[2]{\left\langle #1,~ #2\right\rangle}

\newcommand{\norm}[1]{\left\|#1\right\|}

\newcommand{\und}{\quad\mbox{ and }\quad}

\DeclareMathOperator{\sign}{sign}

\DeclareMathOperator{\argmin}{argmin\,}
\DeclareMathOperator{\dist}{dist \;}

\DeclareMathOperator{\lev}{{lev\,}}

\DeclareMathOperator{\dom}{{dom\,}}
\DeclareMathOperator{\gph}{gph}

\DeclareMathOperator{\hull}{co\,}

\DeclareMathOperator{\Fix}{\mathsf{Fix}\,}

\newcommand{\sd}{\partial}

\def\Pr{\ensuremath{{\mathcal P}_\rho}}
\def\Pinf{\ensuremath{{\mathcal P}_\infty}}

\def\Pp{\ensuremath{{\mathcal P}'}}
\def\Dp{\ensuremath{{\mathcal D}'}}

\newcommand{\Id}{\ensuremath{\operatorname{Id}}}

\newtheorem{Theorem}[thm]{Theorem}

\newtheorem{Corollary}[thm]{Corollary}
\newtheorem{Lemma}[thm]{Lemma}
\def\bc{\begin{center}}
\def\ec{\end{center}}
\def\ni{\noindent}

\def\s2c{\vskip 2cm}
\def\bt{\begin{Theorem}}
\def\et{\end{Theorem}}
\def\bl{\begin{Lemma}}
\def\el{\end{Lemma}}
\def\bcor{\begin{Corollary}}
\def\ecor{\end{Corollary}}
\def\beqn{\begin{equation}}
\def\eeqn{\end{equation}}

\begin{document}
\title{{Local Linear Convergence of the ADMM/Douglas--Rachford Algorithms without Strong Convexity
and Application to Statistical Imaging\footnote{To appear, SIAM Journal on Imaging Sciences}}}
\author{
Timo Aspelmeier\thanks{Institute for Mathematical Stochastics and 
Felix Bernstein Institute for Mathematical Statistics 
in the Biosciences, 
Georg-August Universit\"at G\"ottingen, 37077 G\"ottingen, Germany
}, ~
C. Charitha\thanks{
Institut f\"ur Numerische und Angewandte Mathematik,
Universit\"at G\"ottingen,
37083 G\"ottingen, Germany
}, ~
and D. Russell Luke\thanks{
Institut f\"ur Numerische und Angewandte Mathematik,
Universit\"at G\"ottingen,
37083 G\"ottingen, Germany
}
}
\date{\ttfamily version \today}
\maketitle
\begin{abstract}
We consider the problem of minimizing the sum of a  convex function and 
a convex function composed with an injective linear mapping.  
For such problems, subject to a coercivity condition at fixed points of the 
corresponding  Picard iteration, iterates of the alternating directions method of multipliers converge 
locally linearly to points from 
which the solution to the original problem can be computed.  Our proof strategy uses duality and 
strong metric subregularity of the Douglas--Rachford fixed point mapping.  
Our analysis does not require strong convexity and yields error bounds to the set of model solutions.  
We show in particular that convex piecewise linear-quadratic functions naturally satisfy the 
requirements of the theory, guaranteeing eventual linear convergence of both the Douglas--Rachford algorithm
and the alternating directions method of multipliers for this class of objectives under mild assumptions
on the set of fixed points.  
We demonstrate this result on quantitative image deconvolution and denoising with multiresolution statistical constraints. 
\end{abstract}

{\small \noindent 
Research of T. Aspelmeier, C. Charitha and D. R. Luke  was supported in part by the German Research Foundation grant SFB755-A4. \\
}

{\small \noindent {\bfseries 2010 Mathematics Subject
Classification:} {Primary 49J52, 49M20, 90C26;
Secondary 15A29, 47H09, 65K05, 65K10, 94A08. 
}

\noindent {\bfseries Keywords:}
Augmented Lagrangian, ADMM, Douglas--Rachford, exact penalization, fixed point theory, image processing, inverse problems, metric regularity, 
statistical multiscale analysis, piecewise linear-quadratic, linear convergence 
}

\section{Introduction.}
The alternating directions method of multipliers (ADMM) has received a great deal of attention recently for 
large-scale problems involving constraints on the image of the unknowns under some linear mapping.  
The analysis has focused on either global complexity estimates \cite{HeYuan12} or 
sufficient conditions for local linear convergence \cite{Boley13, NishiharaLessardRechtPackardJordan15, EcksteinYao15}.   
The closely related Douglas--Rachford algorithm has also been the focus of 
recent studies showing global complexity \cite{PatrinosStellaBemporad14, LiPong15} and (local linear) convergence 
in increasingly inhospitable settings
\cite{HesseLuke13, HesseLukeNeumann14, Phan15, BauschkeNoll14, BauschkeBelloCruzNghiaPhanWang14, 
BauschkeNollPhan15,AragonBorweinTam13,BorweinTam15}.   A survey of results on  
proximal methods in general can be found in \cite{ParikhBoyd14}. 
In the convex setting, the convergence studies for 
both ADMM and Douglas--Rachford share a common thread through the well-known duality between these 
algorithms \cite{Gabay83}.  Studies of ADMM frequently invoke strong convexity.  Studies of Douglas--Rachford, on the other 
hand have, until very recently, focused on {\em feasibility} problems and corresponding notions of regularity of 
intersections.  
We combine an analysis of the ADMM algorithm with facts learned from the local 
convergence of Douglas--Rachford to provide sufficient conditions for local linear convergence of sequences generated 
by ADMM without strong convexity.   While this paper was under review we became aware of two recent studies 
that also combine the analysis of ADMM and Douglas-Rachford to improve and generalize many local and global results
\cite{Giselsson15, Giselsson15b}.  
While our theoretical results are general and abstract, our motivation for the current study comes 
from the application of 
statistical multiscale image denoising/deconvolution following \cite{FrickMarnitzMunk12, FrickMarnitzMunk13} 
for fluorescence microscopic images 
(see also \cite{Aspelmeier2015}  for a review of fluorescence microscopy techniques and statistical methods for them). 
We demonstrate the analysis for image denoising/deconvolution of Stimulated Emission Depletion (STED) images 
\cite{Hell1994, Hell2000}.

\subsection{Notation and definitions}
Though many of the arguments presented here work equally well for infinite dimensional Hilbert spaces, 
to avoid technicalities, {\em it will be assumed throughout that $U$ and $V$ are Euclidean spaces.}
The norm $\|\cdot\|$ denotes the Euclidean norm.  
We denote the extended reals by 
$\extre\equiv \Rbb\cup\{+\infty\}$ and the nonnegative orthant by $\Rp\equiv\set{x\in\Rbb}{x\geq 0}$.   
The closed unit ball centered at the origin is denoted by $\Bbb$.  In the usual notation 
for the natural numbers $\Nbb$ we include $0$.  
The mapping $A:U \rightarrow V$ is linear and the functional 
$J: U \rightarrow \extre$ is proper (not everywhere $+\infty$ and nowhere $-\infty$), 
convex and lower semicontinuous (lsc), as is the functional $H:V\rightarrow \extre$. 
The level set of $J$ corresponding to $\alpha\in \Rbb$ is 
defined by $lev_{\leq\alpha}J:=\{u \in U: J(u)\leq \alpha\}$. The domain of a function 
$\map{f}{U}{\extre}$ is defined by $\dom f= \{u\in U: f(u)<\infty$\}.
We use the notation $\mmap{\Phi}{U}{V}$ to denote a set-valued mapping from $U$ to $V$. 

A proper function $\map{f}{U}{\extre}$ is {\em strongly convex} if there is a constant $\mu>0$ 
such that
\begin{equation}\label{e:strong cvx}
  f\left((1-\tau)x_0 +\tau x_1\right)\leq (1-\tau)f(x_0) + \tau f(x_1) - \tfrac12\mu\tau(1-\tau)\|x_0-x_1\|^2,
\end{equation}
for all $x_0$ and $x_1$ and $\tau \in (0,1)$.  
We will not assume smoothness of functions 
and so will require the {\em subdifferential}. 
The 
subdifferential of a function $\map{f}{U}{\extre}$ at a point $\xbar\in\dom f$ is defined by 
\begin{equation}\label{e:sd}
   \sd f(\xbar)\equiv\set{v\in U}{\left\langle v, x-\xbar\right\rangle\leq f(x)-f(\xbar), \mbox{ for all }x\in U}.
\end{equation}
 When $\xbar\notin \dom f$ the subdifferential is defined to be empty.  Elements from the subdifferential 
are called {\em subgradients}.  The subdifferential of a proper, lsc convex function is a 
{\em maximally monotone} set-valued mapping \cite[Theorem 12.17]{VA}.    
The {\em Fenchel conjugate} of a function $f$ is denoted by $f^*$ and defined by 
\[
f^*(y)\equiv\sup_{x\in U}\klam{\ip{y}{x}-f(x)}.  
\]
A mapping $\mmap{\Phi}{V}{V}$ is said to be {\em $\beta$-inverse strongly monotone} 
\cite[Corollary 12.55]{ VA} if for all $x, x' \in V$ 
 \begin{equation}\label{e:beta-inverse str mon}
\langle v-v', x-x'\rangle \geq \beta \|v-v'\|^2, \quad \mbox{whenever}\quad v\in \Phi(x), v'\in \Phi(x').
 \end{equation}
 The mapping $\Phi$ is said to be {\em polyhedral} (or piecewise polyhedral \cite{VA}) 
 if its graph is the union of finitely many sets that 
 are polyhedral convex in $U\times V$ \cite{DontchevRockafellar09}.
We denote the {\em resolvent} of $\Phi$ by  
$\Jcal_{\Phi}\equiv \left(\Id+\Phi\right)^{-1}$ where $\Id$ denotes the identity mapping and 
the inverse is defined as 
\begin{equation}\label{e:inv set}
   \Phi^{-1}(y)\equiv \set{x\in U}{y\in \Phi(x)}.
\end{equation}
The corresponding {\em reflector} is defined by  $R_{\eta \Phi}\equiv 2\Jcal_{\eta \Phi}-\Id$. 

Notions of {\em continuity} of set-valued mappings have been thoroughly developed over the last 
$40$ years.  Readers are referred to the monographs 
\cite{AubinFrankowska90, VA, DontchevRockafellar09} for basic results.
A mapping $\mmap{\Phi}{U}{V}$ is said to be {\em Lipschitz continuous} if it is closed
valued and for all $u, u' \in U$ 
there exists a $\tau\geq 0$ such that
 \begin{equation}
  \Phi(u')\subset \Phi (u) + \tau \|u'-u\| \Bbb. 
 \end{equation}
 Lipschitz continuity is, however, too strong a notion for set-valued mappings. 
A key property of set-valued mappings that we will rely on is {\em metric subregularity}, which 
can be understood as the property corresponding to a Lipschitz-like continuity of the inverse mapping
relative to a specific point.  As the name suggests, it is a weaker property than {\em metric regularity} 
which, in the case of an $n\times m$ matrix for instance, is equivalent to surjectivity. 
Our definition follows the characterization of this property given in 
\cite[Exercise 3H.4]{DontchevRockafellar09}.
\begin{defn}[(strong) metric subregularity]\label{d:metric subregularity}
$~$
\begin{enumerate}[(i)]
 \item\label{d:metric subregularity i} The mapping $\Phi:U\rightrightarrows V$ is called 
\emph{metrically subregular at $\overline x$ for $\overline y$ relative to $W\subset U$} 
if $(\overline x,\overline y)\in \gph\Phi$ and there is a constant $c>0$ and neighborhoods 
$\Ocal$ of $\overline x$ such that
\begin{equation}
\dist(x,\Phi^{-1}(\overline y)\cap W)\leq c\dist(\overline y, \Phi(x))~ \forall~ x\in \Ocal\cap W.
\label{eq:metricsubregularity}
\end{equation}
\item\label{d:metric subregularity ii} The mapping $\Phi$ is called \emph{strongly metrically subregular at 
$\overline x$ for $\overline y$ relative to $W\subset U$} 
if $(\overline x,\overline y)\in \gph\Phi$ and there is a constant $c>0$ and neighborhoods 
$\Ocal$ of $\overline x$ such that
\begin{equation}
\|x-\xbar\|\leq c\dist(\overline y, \Phi(x))~ \forall~ x\in \Ocal\cap W.
\label{eq:strong metricsubregularity}
\end{equation}
\end{enumerate}
\end{defn}
\noindent The constant $c$ measures the stability under perturbations of inclusion $\ybar\in \Phi(\xbar)$. 

An important instance where metric subregularity comes for free is for polyhedral mappings.  
\begin{propn}[polyhedrality implies strong metric subregularity]\label{t:polyhedrality-strong msr}
Let $W\subset V$ be an affine subspace and $\mmap{T}{W}{W}$.  If $T$ is polyhedral and 
$\Fix T\cap W$ is an isolated point, $\{\xbar\}$, 
then $\mmap{\Id-T}{W}{(W-\xbar)}$ is strongly metrically 
subregular, hence metrically subregular, at $\xbar$ for $0$ relative to $W$.     
\end{propn}
\begin{proof}
If $T$ is polyhedral, so is 
$\Phi^{-1}\equiv (\Id-T)^{-1}$.  Now by 
\cite[Propositions  3I.1 and 3I.2]{DontchevRockafellar09}, since 
$\Phi^{-1}$ is polyhedral and $\xbar$ is an isolated point of $\Phi^{-1}(0)\cap W$,
then $\Phi= \Id - T$ is {\em strongly metrically subregular}  at $\xbar$ for $0$ with constant $c$ on 
the neighborhood $\Ocal$ of $\xbar$ restricted to 
$W$ \eqref{eq:strong metricsubregularity}.  
\end{proof}

\noindent One prevalent source of polyhedral mappings is the subdifferential of piecewise linear-quadratic functions (see 
Proposition \ref{t:polyhedral subdiff} below).
\begin{defn}[piecewise linear-quadratic functions]\label{d:plq}
 A function $f:\Rn \rightarrow [-\infty, +\infty]$ is called \emph{piecewise linear-quadratic} 
 if $dom f$ can be represented as the union
of finitely many polyhedral sets, relative to each of which $f(x)$ is given by an expression of the form 
$\frac{1}{2}\langle x, Ax\rangle +\langle a, x \rangle+\alpha $ for some scalar $\alpha\in \mathbb{R}$ vector
$a\in \Rn$, and symmetric matrix $A\in \mathbb{R}^{n\times n} $. 
\end{defn}

A notion related to metric regularity is that of {\em weak-sharp solutions}.  This will be used in the development 
of error bounds (Theorem \ref{t:dist to Sinf}).  
\begin{defn}[weak sharp minimum \cite{BurkeFerris93}]\label{d:weak sharp} 
The solution set $\argmin\set{f(x)}{x\in\Omega}$ for a nonempty closed convex set 
$\Omega$, is weakly sharp if, for $\pbar=\inf_\Omega f$, there exists a positive number $\alpha$ (sharpness constant) such that 
\begin{equation*}
 f(x)\geq \pbar+\alpha~\dist(x, S_f)~~~\forall x\in \Omega.
\end{equation*}
Similarly, the solution set $S_f$ is weakly sharp 
of order $\nu>0$ if there exists a positive number $\alpha$ (sharpness constant) such that, 
for each $x\in\Omega$, 
\begin{equation*}
 f(x)\geq \pbar+\alpha~\dist(x, S_f)^\nu~~~\forall x\in \Omega.
\end{equation*} 
\end{defn}

\subsection{Preparatory abstract results}\label{s:abstract}
To conclude this section we present general results about types of (firmly) nonexpansive operators
that clarify the underlying mechanisms yielding linear convergence of many algorithms.   
The operative definitions are given here. 
\begin{defn}[$(S,\epsilon)$-(firmly-)nonexpansive mappings]\label{d:q(f)ne}
Let $D$ and $S$ be nonempty subsets of $U$ and let $T$ be a (multi-\-valued) mapping from $D$ to $U$.
\begin{enumerate}[(i)]
   \item\label{d:qne}  $T$ is called \emph{$(S,\varepsilon)$-nonexpansive on $D$} if 
\begin{eqnarray}\label{eq:epsqnonexp}
&&\norm{x_+-\xbar_+}\leq\sqrt{1+\varepsilon}\norm{x-\xbar},\\
&&\forall x\in D,~\forall \xbar\in S,~\forall x_+\in Tx,~\forall \xbar_+\in T\xbar. \nonumber
\end{eqnarray}
If \eqref{eq:epsqnonexp} holds with $\epsilon=0$ then we say that $T$ is \emph{$S$-nonexpansive} on $D$.
\item\label{d:qfne} $T$ is called \emph{$(S,\varepsilon)$-firmly nonexpansive} on $D$ if  
\begin{eqnarray}\label{eq:quasifirm}
&&\norm{x_+-\xbar_+}^2+\norm{(x -x_+)-(\xbar -\xbar_+)}^2\leq(1+\varepsilon)\norm{x-\xbar}^2,\\
&&\forall x\in D,~\forall \xbar\in S,~\forall x_+\in Tx,~\forall \xbar_+\in T\xbar. \nonumber
\end{eqnarray}
If \eqref{eq:quasifirm} holds with $\epsilon=0$ then we say that $T$ is \emph{$S$-firmly nonexpansive} on $D$.
If, in addition,  $S=\Fix T$, then $T$ is said to be {\em quasi-firmly nonexpansive}. 
\end{enumerate}
\end{defn}

%

\begin{thm}[abstract linear convergence result]\label{t:gen convergence}
Let $W\subset V$ be an affine subspace and $\mmap{T}{W}{W}$ be quasi-firmly nonexpansive on $W$. 
Let $\Fix T\cap W$ be an isolated point, $\{\xbar\}$.  
If  $\mmap{\Id-T}{W}{(W-\xbar)}$ is metrically subregular at $\xbar$ for $0$, then there
is a neighborhood $\Ocal$ of $\xbar$  such that
\begin{equation}\label{e:gen convergence}
 \dist(x_+, \Fix T)\leq \sqrt{1-\kappa}\dist(x, \Fix T),\quad\forall x_+\in Tx, ~\forall x\in \Ocal\cap W,
\end{equation}
where $0<\kappa=c^{-2}$ for $c$ a constant of metric subregularity of $\Id-T$ at $\xbar$ for the 
neighborhood $\Ocal$.   
Consequently, the fixed point iteration $x^{k+1}=Tx^k$ converges linearly to $\Fix T$ with rate 
$\sqrt{1-\kappa}$ for all $x^0\in\Ocal\cap W$. 
\end{thm}
\begin{proof}
Define $\Phi\equiv (\Id-T)$ and note that $\{\xbar\}=(\Id-T)^{-1}(0)\iff \{\xbar\}=\Fix T$, 
hence 
\[
\dist(x, (\Id-T)^{-1}(0)) = \dist(x, \Fix T) = \|x-\xbar\|. 
\]
Suppose that $\Phi$ is metrically subregular at $\Fix T$ for $0$.  Then by 
Definition \ref{d:metric subregularity}\eqref{d:metric subregularity i} we have, 
for all $x\in \Ocal\cap W$ and for all $x^+\in T(x)$,
\begin{eqnarray}\label{e:abs coercive}
\dist(x, (\Id-T)^{-1}(0)) =  \|x-\xbar\| &\leq& c \dist\left(0, (x-Tx)\right) \leq c \|x-x^+\|,
\end{eqnarray}
which is the coercivity condition of \cite[Eq.(3.1), Lemma 3.1]{HesseLuke13}.
By assumption, $T$ is 
$(\Fix T,0)$-firmly nonexpansive (i.e., quasi-firmly nonexpansive) on $W$ 
(Definition \ref{d:q(f)ne} \eqref{d:qfne}).  The result then follows from 
\cite[Lemma 3.1]{HesseLuke13} with rate $\sqrt{1-\kappa}$ for $\kappa=c^{-2}$. 
\end{proof}
\begin{remark}[on $\kappa$]
   The constant $\kappa$ in the above theorem can always be chosen to be less than or equal to $1$.
To see this, note that for any metrically subregular mapping $\Phi$, there is a constant $c\geq 1$ and hence a 
$\kappa\leq 1$ so that the rate constant given in Theorem \ref{t:gen convergence} will always hold whenever the 
fixed point is a 
(relatively) isolated point. 
\end{remark}
\begin{eg}[a simple example]\label{eg:example}
Consider two lines, $A$ and $B$, in $\Rtw$ 
intersecting orthogonally at the origin and let $T$ be the Douglas--Rachford operator for the projections onto each line. 
In this example $T=\tfrac12\paren{R_AR_B+\Id}$ where $R_A\equiv 2P_A-\Id$ for the projection 
onto the line $A$ denoted by $P_A$, and likewise for $R_B$.  In the context of what follows, $P_A$ is the resolvent of 
the subdifferential of the indicator function of the line $A$ and likewise for $P_B$.  It is elementary to verify 
that $T$ is firmly nonexpansive, has a unique fixed point, and $T(x)=0$ for all $x$. Moreover  
$\Phi=\Id-T=\Id$, which has a constant of metric subregularity $c=1$.  Theorem \ref{t:gen convergence} then predicts that 
the Douglas--Rachford algorithm converges linearly with rate constant $0$ in this case, i.e. it converges in one step. 
The reader can verify that this indeed is the case.  

To see the importance of the restriction to the affine subspace $W$, consider instead of  two lines in 
$\Rtw$ two lines in $\Rth$ intersecting at the origin.  It can be shown that 
 the fixed points of the Douglas--Rachford operator consist of the axis -- let's call 
 it the $z$ axis -- extending 
 from the origin, perpendicular to the linear hull of the two lines \cite{BauschkeCombettesLuke04}.
 It is elementary to verify that, from any starting point $x^0$ in $\Rth$, the  
 Douglas--Rachford algorithm converges in one step to the intersection 
 of the $z$ axis with the affine subspace containing $x^0$ and parallel to the 
 plane containing the lines $A$ and $B$.    Clearly, the fixed points of the mapping 
 $T$ are not isolated points, but they are isolated points relative to the affine subspace
 containing $x^0$ and parallel to $A$ and $B$, so Theorem \ref{t:gen convergence} applies
 and predicts, correctly, that the Douglas-Rachford algorithm converges to a fixed point in one step.  
 The projection of this fixed point onto the set $B$ is the solution to the problem of finding the 
 point of intersection. 
\end{eg}

\noindent 

\begin{cor}[Polyhedrality implies linear convergence]\label{t:polyhedral convergence}
 Let $W\subset V$ be an affine subspace and $\mmap{T}{W}{W}$ be quasi-firmly nonexpansive on $W$. 
Let $\Fix T\cap W$ be an isolated point, $\{\xbar\}$.  
If  $T$ is polyhedral, then there
is a neighborhood $\Ocal$ of $\xbar$ such that
\[
\dist(x_+, \Fix T)\leq \sqrt{1-\kappa}\dist(x, \Fix T)\quad\forall x_+\in Tx, ~\forall x\in \Ocal\cap W,   
\]
where $0<\kappa=c^{-2}$ for $c$ a constant of metric subregularity of $\Id-T$ at $\xbar$ for the 
neighborhood $\Ocal\cap W$.   
Consequently, the fixed point iteration $x^{k+1}=Tx^k$ converges linearly to $\Fix T$ with rate 
$\sqrt{1-\kappa}$ for all $x^0\in\Ocal\cap W$. 
\end{cor}
\begin{proof}
The result follows immediately from Proposition \ref{t:polyhedrality-strong msr} and Theorem \ref{t:gen convergence}.
\end{proof}

\noindent The requirement that the fixed point set is a singleton can be viewed as a uniqueness assumption, which 
is common in the inverse problems literature.  
It is well known, however, that, even if the solution to a given problem is unique, the 
set of fixed points of the numerical method (of interest to us, the Douglas--Rachford 
operator)  need not be solutions to the given problem, much less be unique
\cite{LionsMercier79, BauschkeCombettesLuke04}.  Recent work  has shown, however, 
that the set of fixed points need only consist of singletons relative to 
appropriate affine subspaces where the iterates lie \cite{Phan15, HesseLukeNeumann14}.  This feature has been
exploited in the analysis of the Douglas--Rachford algorithm applied to problems with polyhedral and quadratic structure
\cite{LiangPeyreFadiliLuke}.  
Metric (sub)regularity, on the other hand,  is one of the central assumptions of well-posedness of inverse 
problems \cite{KlatteKummer09, DontchevRockafellar09}.  
Other useful equivalent characterizations of metric subregularity can be found in \cite{DontchevRockafellar09}.  
Polyhedrality can be quite easy to verify, as we will see below.

\section{Linear Convergence of Douglas--Rachford/\\ Alternating Directions Method of Multipliers}\label{s:siadpmm}
We consider problems in the following format:  
\begin{equation*}\tag{\Pp}\label{e:Pp}
\ucmin{ J(u)+H(Au)}{u\in U}.   
\end{equation*}
There are many possibilities for solving such problems. We focus our attention on one of the more 
prevalent methods, the  alternating direction method of multipliers, abbreviated
as ADMM (primary sources include 
\cite{Rockafellar76, Gabay83, Eckstein, EckBert92, Glowinski75}).
This method is one of many {\em splitting methods} which are the principle approach to
handling the computational burden of large-scale, 
separable problems \cite{BoydParikhChuPeleatoEckstein11}.
ADMM belongs to a class of {\em augmented Lagrangian methods} whose original 
motivation was to regularize Lagrangian formulations of constrained optimization problems. 

Introducing a new variable $v\in V$, our problem is to solve
\begin{equation}\label{problemwithAu=v}
\ucmin{ J(u)+H(v)}{(u,v)\in U\times V},~~~\mbox{ subject to }  Au=v. 
\end{equation}

The  augmented Lagrangian $\Ltilde$ for \eqref{problemwithAu=v} is given by 
\begin{equation}
 \Ltilde(u,v,b)=J(u)+H(v)+\langle b, Au-v\rangle+\tfrac{\eta}{2}\|Au-v\|^2,
\end{equation}
where $b\in V$, $\eta>0$ is a fixed penalty parameter.
The ADMM algorithm for solving \eqref{problemwithAu=v} is, given $(u^k, v^k, b^k)$, $k\in\Nbb$, 
compute $(u^{k+1}, v^{k+1}, b^{k+1})$ by
\begin{eqnarray}
 \label{ADMM1'}u^{k+1}&\in&
  \argmin_u\left\{J(u)+\tfrac{\eta}{2}\|Au-v^k+\eta^{-1}b^k\|^2\right\};\\
 \label{ADMM2'}v^{k+1}&\in &
  \argmin_v\left\{H (v)+\tfrac{\eta}{2}\|Au^{k+1}-v+\eta^{-1}b^k\|^2\right\};\\
 \label{ADMM3'}b^{k+1}&=&
  b^{k}+\eta( Au^{k+1}- v^{k+1}).
\end{eqnarray}
Using $\tfrac{\eta}{2}\|Au-v+\eta^{-1}b^k\|^2-\tfrac{1}{2\eta}\|b^k\|^2=\langle b^k, Au-v\rangle+\tfrac{\eta}{2}\|Au-v\|^2$, 
the algorithm 
\eqref{ADMM1'}-\eqref{ADMM3'} can be written equivalently as 
\begin{center}
\fbox{%
        \addtolength{\linewidth}{-2\fboxsep}%
        \addtolength{\linewidth}{-2\fboxrule}%
    \begin{minipage}{\linewidth}%
	\begin{algorithm}[{\bf ADMM}] \label{a:ADMM}$~$\\
	    {\bf Initialization.} Choose $\eta > 0$ and $(v^{0}, b^{0})\in U\times V\times V$. \\
	    {\bf General Step ($k = 0 , 1 , \ldots$)}
\begin{subequations}
	    \begin{eqnarray}
		\label{ADMM1} u^{k+1}&\in&\argmin_{u}\klam{J(u) +\langle b^k, Au \rangle+\tfrac{\eta}{2}\|Au-v^k\|^2};\\
		\label{ADMM2} v^{k+1}&\in&
			  \argmin_{v}\klam{H(v)-\langle b^k, v\rangle+\tfrac{\eta}{2}\|Au^{k+1}-v\|^2};\\
		\label{ADMM3} b^{k+1}&=& b^k+\eta(Au^{k+1}-v^{k+1}).
\end{eqnarray}
\end{subequations}
	\end{algorithm}
    \end{minipage}%
}
\end{center}
The penalty parameter $\eta$ need not be a constant, and indeed evidence indicates that the choice of $\eta$ can greatly 
impact the 
complexity of the algorithm, but this is beyond the scope of this investigation, so we have left this parameter fixed.  

We do not specify how the argmin in steps \eqref{ADMM1}-\eqref{ADMM2} should be calculated, and indeed, 
the analysis that follows assumes that these can be computed exactly.  This is, 
of course, not true in practice.  In an attempt to circumvent this fact, the standard approach in 
numerical analysis is to accommodate summable errors.  The generalization to summable errors is, however,   
tantamount to {\em eventual} exact evaluation of \eqref{ADMM1}-\eqref{ADMM2} and thus, for all practical 
purposes, is no different from {\em immediate} exact evaluation, the latter involving errors that sum to zero. 

Even if we do assume infinite precision, a few remarks about the computational complexity of the individual 
steps of Algorithm \ref{a:ADMM} are warranted.  Inspection of 
\eqref{ADMM1} shows that an implicit method involving computation of the inverse of $A^TA$ may not be 
feasible if this is very large or does not otherwise enjoy a structure that allows for efficient inversion.  
If $J$ is smooth, a number of classical quasi-Newton methods, with error bounds, are available \cite{NocedalWright}.  
If $J$ is nonsmooth, then a forward-backward-type method such as FISTA \cite{BeckTeboulle09} could be applied.  
In the latter case new results on convergence of the iterates to a solution open the door to error bounds at this stage 
\cite{AttouchePeypouquet15}.  The second step \eqref{ADMM2} does not involve any matrix inversion, but will, 
for exact penalization, involve a nonsmooth penalty $H$.  Again, one has recourse to fast first-order methods that,
as of very recently, permit error bounds.

Our goal is to determine the rate of convergence 
of these algorithms so that they may be used as inner routines in an iteratively 
regularized procedure.  Knowing that an algorithm converges linearly, 
for instance, yields rational stopping criteria with computable estimates for the distance of 
the current iterate to the solution set.  

We present sufficient conditions for  {\em linear} convergence of 
Algorithm \ref{a:ADMM} by showing the same for the Douglas-Rachford algorithm
which is more amenable to the tools of abstract fixed point theory presented in Section \ref{s:abstract}. 
It is well known \cite{Gabay83, Eckstein} that the ADMM algorithm 
can be derived from the Douglas--Rachford algorithm, and vice versa, and therefore sufficient conditions for 
convergence of Douglas--Rachford also apply here.  
The first convergence result for Douglas--Rachford is due to Lions and Mercier \cite{LionsMercier79}, under the 
assumption of strong convexity and Lipschitz continuity of $J$.
Recent published work in this direction includes \cite{HeYuan12, Giselsson15, Giselsson15b}. 
Convergence rates with respect to objective values under various assumptions on the objective, all of which involving 
strong convexity, was established in \cite{HeYuan12, GoldsteinO'DonoghueSetzer12} 
which is conservative.  Local linear convergence of the iterates to a {\em solution} was established in \cite{Boley13} for 
linear and quadratic programs using spectral analysis.  In the first main result, Theorem \ref{t:convergence}, we describe 
two conditions that guarantee linear 
convergence of the ADMM iterates to a solution.  The first of these conditions follows from classical 
results of Lions and Mercier \cite{LionsMercier79}.  
The second condition is based on work of more recent vintage \cite{HesseLuke13}, is much more prevalent in 
applications and generalizes the results of \cite{Boley13}.  

The (Fenchel-Legendre) dual problem corresponding to the problem \eqref{e:Pp} is
(see, for instance \cite{CANO2})
 \begin{equation*}
\min_{w\in V}J^*(A^{T}w)+H^*(-w).
\end{equation*}
Here $J^*$ and $H^*$ are the Fenchel conjugates of $J$ and $H$ respectively.
Instead of working with this dual, we work with the following equivalent form with the change of variable $v=-w$:
 \begin{equation*}\tag{\Dp}\label{e:Dp}
\min_{v\in V}J^*(-A^{T}v)+H^*(v).
\end{equation*}
Under the assumption that the solutions $\ubar$ and $\bbar$ of the primal and dual problems exist and that the 
dual gap is zero, 
the following two inclusions characterize the solutions of the  problems \eqref{e:Pp} and 
\eqref{e:Dp} respectively:
\[
0 \in \partial J(\ubar)+\partial (H\circ A)(\ubar);
\]
\[
0 \in \partial \left(J^*\circ (-A^{T})\right)(\bbar)+\partial H^* (\bbar).
\]
In both cases, one has to solve an inclusion of the form
\begin{equation}\label{problem inclusion}
0 \in (B+D)(x),
\end{equation}
for general set-valued mappings $B$ and $D$.
For any $\eta>0$, the Douglas--Rachford algorithm \cite{DougRach56, LionsMercier79} 
for solving (\ref{problem inclusion}) is given by
\begin{eqnarray}\label{e:DRSalg1}
 &b^{k+1}&\in T'b^k\quad (k\in\Nbb),\\
\mbox{ for } &T'&\equiv \Jcal_{\eta D}\left(\Jcal_{\eta B}(\Id-\eta D)+\eta D\right),
\label{e:T'}
\end{eqnarray}
where  $\Jcal_{\eta D}$ and $\Jcal_{\eta B}$ are the {\em resolvents} of $\eta D$ and $\eta B$ respectively. 
The connection between the ADMM  algorithm \eqref{ADMM1}-\eqref{ADMM3} 
and  the Douglas--Rachford algorithm \eqref{e:DRSalg1} was first discovered by Gabay \cite{Gabay83} and is 
derived for convenience in the Appendix.  

Given $b^0$ and $v^0\in Db^0$, following \cite{Svaiter11}, define the new variable 
$x^0\equiv b^0+\eta v^0$ so that $b^0= \Jcal_{\eta D}x^0$.  
We thus arrive at an alternative formulation of the Douglas--Rachford algorithm \eqref{e:DRSalg1}:
\begin{eqnarray}\label{e:DRSalg2}
 &x^{k+1}&\in Tx^k \quad (k\in\Nbb),\\
\mbox{ for } &T&\equiv \tfrac{1}{2}(R_{\eta B}R_{\eta D}+\Id)=\Jcal_{\eta B}(2\Jcal_{\eta D}-\Id)+(\Id-\Jcal_{\eta D}),
\label{e:T_DR}
\end{eqnarray}
where $R_{\eta D}$ and $R_{\eta B}$ are the reflectors of the respective resolvents.
This is exactly the form of Douglas--Rachford considered in \cite{LionsMercier79}.  
\begin{remark}[proximal mappings of convex functions]\label{r:resolvents sv}
Note that for our application 
\begin{equation}\label{e:BnD}
B\equiv\partial \left(J^*\circ(-A^{T})\right)\und D\equiv\partial H^*,
\end{equation}
and so the 
resolvent mappings are the proximal mappings of the convex functions $\left(J^*\circ(-A^{T})\right)$ 
and $H^*$ respectively, and hence the resolvent mappings and corresponding fixed point operator $T$ 
are single-valued  \cite{Moreau65}.  
\end{remark}

\begin{propn}\label{t:pre convergence} 
 Let $J:U\to \Rbb\cup\{+\infty\}$ and $H:V\to \Rbb$ 
be proper, lsc and convex.  
Let $\map{A}{U}{V}$ be linear 
and suppose there exists a solution to $0\in  (B+D)(x)$ for 
$B$ and $D$ defined by \eqref{e:BnD}. For fixed $\eta>0$, given any 
initial points $x^0$ and $\paren{b^0, v^0}\in \gph D$ such that $x^0=b^0+\eta v^0$, 
the sequences $\paren{b^k}_{k\in\Nbb}$, $\paren{x^k}_{k\in\Nbb}$ and $\paren{v^k}_{k\in\Nbb}$  
defined respectively by \eqref{e:DRSalg1}, \eqref{e:DRSalg2} and $v^k\equiv\tfrac1\eta\paren{x^k-b^k}$
converge to points $\bbar\in\Fix T'$, $\xbar\in\Fix T$ and $\vbar\in D\paren{\Fix T'}$.
The point $\bbar = \Jcal_{\eta D}\xbar$ is a solution to \eqref{e:Dp},  
and $\vbar=\frac{1}{\eta}\left(\xbar-\bbar \right)\in D\bbar$.  If, in addition, $A$ has full column rank, 
then  the sequence $\paren{b^k, v^k}_{k\in\Nbb}$ corresponds exactly to the sequence of points generated in 
steps \eqref{ADMM2} and \eqref{ADMM3} of Algorithm \ref{a:ADMM} and the sequence $\paren{u^{k+1}}_{k\in\Nbb}$
generated by \eqref{ADMM1} 
converges to $\ubar$, a solution to \eqref{e:Pp}.
\end{propn}
\begin{proof}
   Following \cite{Eckstein, Svaiter11}, we rewrite the Douglas--Rachford iteration \ref{e:DRSalg1} in two steps:
Given $(b^0, v^0)\in\gph D$, for $k\in\Nbb$ do 
\begin{subequations}
\label{a:DRSalg1'}
\begin{eqnarray}
 \label{a:DRSalg1'a} &\mbox{find }(q^{k+1}, s^{k+1})\in \gph(B)\mbox{ such that }q^{k+1}+\eta s^{k+1}=b^k-\eta v^k;&\\
 \label{a:DRSalg1'b} &\mbox{find }(b^{k+1}, v^{k+1})\in \gph(D) \mbox{ such that }b^{k+1}+\eta v^{k+1}=q^{k+1}+\eta v^k.&
\end{eqnarray}
\end{subequations}
The existence and uniqueness in the above steps follows from the representation lemma \cite[Corollary 3.6.3]{Eckstein}. 
The mappings $B,D$ are  maximal 
monotone operators as the subdifferentials of proper lsc convex functions.  This together with the 
fact that the solution set of (\ref{problem inclusion}) is non-empty yields that
the sequence $(b^k, v^k)_{k\in \Nbb}$ defined by the algorithm \eqref{a:DRSalg1'} converges to some 
$(\bbar, \vbar)$ such that $\vbar\in D\bbar$ and $\bbar$ solves \eqref{e:Dp} \cite[Theorem 1]{Svaiter11}.  
By the change of variables 
$x^k=b^k + \eta v^k$, it follows that $x^k\to \xbar\in \Fix T$ for $T$ given by \eqref{e:T_DR}.

For these definitions of $B$ and $D$, the sequence $\paren{b^k}_{k\in\Nbb}$ generated by 
$b^k\equiv\Jcal_{\eta D}x^k$ for $x^k$ generated by \eqref{e:DRSalg2} corresponds exactly to the sequence 
$\paren{b^k}_{k\in\Nbb}$ generated by \eqref{e:DRSalg1}.  Moreover, if $A$ is full column rank, then 
by the discussion in \cite{Eckstein} (see the Appendix) both $\paren{b^k}_{k\in\Nbb}$ and 
the sequence $\paren{v^k}_{k\in\Nbb}$ generated by $v^k\equiv \tfrac{1}{\eta}\paren{x^k-b^k}\in Db^k$ 
correspond exactly to the sequences of points $b^k$ and $v^k$ generated by \eqref{ADMM1}-\eqref{ADMM3}. 
Consequently, by \cite[Proposition 3.42]{Eckstein}\footnote{By convergence of $v^k\to \vbar$ and $b^k\to \bbar$ 
and the update rule
\eqref{ADMM3}, $Au^k\to \vbar$, from which the claim follows -- see the Appendix.} 
the sequence $\paren{u^k}_{k\in\Nbb}$ defined by \eqref{ADMM1} converges to a solution of \eqref{e:Pp}. 
\end{proof}

We now state sufficient conditions guaranteeing linear convergence of the ADMM and the Douglas--Rachford algorithms.  
The first conditions \eqref{t:convergence_i}
of Theorem \ref{t:convergence} are classical.  The second conditions are new. 
\begin{thm}[local linear convergence I]\label{t:convergence} 
 Let $J:U\to \Rbb\cup\{+\infty\}$ and $H:V\to \Rbb$ 
be proper, lsc and convex.  Suppose there exists a solution to $0\in (B+D)(x)$  
for  $B$ and $D$ defined by \eqref{e:BnD} where 
$\map{A}{U}{V}$ is an injective linear mapping.  Let $\xhat\in \Fix T$ for $T$ defined
by \eqref{e:T_DR}.  
For fixed $\eta>0$ and any given triplet of points $\paren{b^0, v^0, x^0}$  satisfying 
$x^0\equiv b^0+\eta v^0$, with 
$v^0\in Db^0$, generate the sequence $(v^k, b^k)_{k\in\Nbb}$ by \eqref{ADMM1}-\eqref{ADMM3} 
and the sequence $(x^k)_{k\in\Nbb}$ by \eqref{e:DRSalg2}.   
\begin{enumerate}[(i)]
 \item\label{t:convergence_i} Let $\Ocal\subset U$ be a neighborhood of $\xhat$ on which 
 $H$ is strongly convex with constant $\mu$ and $\sd H$ is $\beta$-inverse strongly monotone
for some $\beta>0$.  
Then, for any $\paren{b^0, v^0, x^0}\in \Ocal$  satisfying $x^0\equiv b^0+\eta v^0\in\Ocal$,
the sequences $(x^k)_{k\in\Nbb}$ and $(v^k, b^k)_{k\in\Nbb}$ converge linearly to the 
respective points $\xbar\in \Fix T$ and $\paren{\bbar,\vbar}$
with rate at least $K=(1-\frac{2\eta\beta\mu^2}{(\mu+\eta)^2})^{\tfrac{1}{2}}<1$. 
 \item\label{t:convergence_ii}  Suppose that $\map{T}{W}{W}$ for some 
 affine subspace $W\subset U$ with $\xhat\in W$.  On the neighborhood $\Ocal$ of $\xhat$ relative
 to $W$, that is $\Ocal\cap W$,  suppose there is a constant 
$\kappa>0$ such that  
\begin{equation}\label{e:coercive}
   \|x-x^+\|\geq \sqrt{\kappa} \dist(x,\Fix T)\quad\forall x\in \Ocal\cap W,~\forall x^+\in Tx.
\end{equation}
Then the sequences $(x^k)_{k\in\Nbb}$ and $(v^k, b^k)_{k\in\Nbb}$ converge linearly to 
the respective points $\xbar\in \Fix T\cap W$ and $\paren{\bbar,\vbar}$
with rate bounded above by $\sqrt{1-\kappa}$.
 \end{enumerate}
In either case, the limit point $\bbar=\Jcal_{\eta D}\xbar$ is a solution to 
\eqref{e:Dp}, $\vbar\in D\bbar$ and the sequence $\paren{u^k}_{k\in\Nbb}$ 
given by \eqref{ADMM1} of Algorithm \ref{a:ADMM} converges to 
$\ubar$, a solution of \eqref{e:Pp}.
 \end{thm}
\begin{proof}
The final statement of the theorem  and the statements about the sequence $\paren{b^k, v^k}$ 
follows from Proposition \ref{t:pre convergence} where it is shown that 
the sequence $(v^k, b^k)_{k\in\Nbb}$ generated 
by \eqref{ADMM1}-\eqref{ADMM3} corresponds to sequences $\paren{b^k}_{k\in\Nbb}$ 
and $\paren{v^k}_{k\in\Nbb}$ generated respectively by \eqref{e:DRSalg1} and 
$v^k=\tfrac1\eta\paren{x^k-b^k}\in D b^k$
for $\paren{x^k}_{k\in\Nbb}$ generated by \eqref{e:DRSalg2}.  The linear convergence 
of the iterates of Algorithm \ref{a:ADMM} claimed in statements \eqref{t:convergence_i}
and \eqref{t:convergence_ii} follows from the properties of the operators $T'$ and 
$T$ defined respectively by \eqref{e:T'} and \eqref{e:T_DR}.

Part \eqref{t:convergence_i}. 
Since $H$ is assumed to be strongly convex with  $\mu>0$ the modulus of convexity on $\Ocal$, 
$\partial H$ is strongly monotone with modulus of monotonicity $\mu$
\cite[Example 22.3]{BauschkeCombettes11}. Since $\sd H$ is also maximally monotone, using the identity 
$\sd H=(\sd H^*)^{-1}$ (see, for example, \cite[Corollary 3.49]{PenotCWD}) we conclude that  $\partial H^*$ is Lipschitz continuous
with constant $\frac{1}{\mu}$. 
Moreover, since $\sd H$ is $\beta$-inverse strongly monotone on $\Ocal$, we have for any $x,y \in \Ocal$
\[\langle u-v, x-y\rangle \geq \beta \|u-v\|^2, \quad \mbox{whenever}\quad
u\in \sd H(x), v\in \sd H(y).\] Hence  $\sd H^*$ is strongly monotone with modulus $\beta$ and Proposition 4 of \cite{LionsMercier79}
applies to yield linear convergence of the sequences $\paren{x^k}$ and $\paren{b^k}$ to the respective 
limit points $\xbar$ and $\bbar$ 
 \begin{equation}
  \|x^{k}-\xbar\|\leq L K^k;~~~\|b^{k}-\bbar\|\leq L K^k,
 \end{equation}
where $L$ is some constant,  $K=(1-\frac{2\eta\beta}{(1+\eta \xi)^2})^{\tfrac{1}{2}}$ and $\xi=\frac{1}{\mu}$ is the Lipschitz
 constant for the set-valued map $\sd H^*$ on $\Ocal$.
Now, since $v^{k}=\tfrac{1}{\eta}(x^{k}-b^{k})$, we have for $v^k\to \vbar:=\tfrac{1}{\eta}(\xbar-\bbar)$ with the same
rate as $x^k$ and $b^k$, modulo a constant:
 \begin{equation}
  \|v^{k}-\vbar\|\leq \tfrac{1}{\eta} \left(  \|x^{k}-\xbar\|+ \|\bbar-b^{k}\|\right)\leq\frac{2L K^k}{\eta}.
 \end{equation}
This completes the proof of the first statement. 
\hfill$\triangle$

 Part \eqref{t:convergence_ii}.
 Since $B$ and $D$ are maximal monotone operators the reflected resolvents 
$R_{\eta B}$ and $R_{\eta D}$ are
nonexpansive \cite[Proposition 23.7]{BauschkeCombettes11}. The composition $R_{\eta B} R_{\eta D}$ 
is nonexpansive which implies that the mapping $T$ is firmly nonexpansive 
\cite[Proposition 4.2]{BauschkeCombettes11}, and hence quasi-firmly nonexpansive on $W$.  
Condition \eqref{e:coercive} 
is the coercivity condition (b) of \cite[Lemma 3.1]{HesseLuke13} which guarantees local linear convergence of 
fixed-point iterations for $(S,\epsilon)$-firmly nonexpansive mappings ($S\subset\Fix T\cap W$).
Quasi-firmly nonexpansive mappings, under consideration here, are $(\Fix T\cap W,0)$-firmly nonexpansive. 
Thus, by  \cite[Lemma 3.1]{HesseLuke13} the sequence $(x^k)_{k\in\Nbb}$ converges linearly on the 
neighborhood $\Ocal$ with rate  $\sqrt{1-\kappa}$.  Nonexpansiveness of the resolvent $\Jcal_{\eta D}$
and the relations $b^k=\Jcal_{\eta D}x^k$ and $v^k=\tfrac1\eta\paren{x^k-b^k}$ then complete the proof of 
the second statement.
\end{proof}
 
\begin{remark}
   The strong convexity assumption \eqref{t:convergence_i} of Theorem \ref{t:convergence} fails in a wide range of
   applications, and in particular for feasibility problems (minimizing the sum of indicator functions).  
By Theorem \ref{t:gen convergence}, case \eqref{t:convergence_ii} of Theorem \ref{t:convergence}, in contrast, holds in 
general for mappings $T$ for which $\Id-T$ is {\em metrically subregular} and the fixed point sets are {\em isolated points} 
with respect to an affine subspace to which the iterates are confined.   The restriction to the affine subspace $W$ 
is a natural generalization for the Douglas--Rachford algorithm, where the iterates are known to stay confined to  
affine subspaces orthogonal to the fixed point set \cite{HesseLukeNeumann14, Phan15}.  It would be far too restrictive 
to require that $\Fix T$ be a singleton on the entire ambient space $V$ rather than with respect to just the 
affine hull of the iterates.   We show that metric subregularity with respect to this affine subspace holds in 
many applications.   (See also Example \ref{eg:example}.)
\end{remark}

\begin{remark}
Proposition \ref{t:pre convergence} and Theorem \ref{t:convergence} and their proofs also hold in infinite 
dimensional Hilbert spaces.  
Lemma 3.1 of \cite{HesseLuke13} is stated for Euclidean spaces, but the proof holds also on general Hilbert spaces.  
\end{remark}

\begin{propn}[polyhedrality of the Douglas--Rachford operator]\label{t:polyhedral subdiff}
Let $J:U\to \Rbb\cup\{+\infty\}$ and $H:V\to \Rbb$ 
be proper, lsc and convex.  
Suppose, in addition, that $J$ and $H$ are piecewise linear-quadratic. 
The operator $\map{T}{V}{V}$ defined by \eqref{e:T_DR} with $\eta>0$ fixed, is polyhedral
for $B$ and $D$ given by \eqref{e:BnD} where
$\map{A}{U}{V}$ is a  linear mapping.  
\end{propn}
\begin{proof}
Since the functions $J$ and $ H$ are proper, lsc, convex and piecewise linear-quadratic, by \cite[Theorem 11.14]{VA} so are
the Fenchel conjugates, $J^*$ and $H^*$.  The subdifferentials 
$B\equiv\partial \left(J^*\circ(-A^{T})\right)$  and $D\equiv\partial H^*$ and their resolvents, therefore, 
are polyhedral mappings \cite[Proposition 12.30]{VA}. Since the graphs of reflectors $R_{\eta B}$ and
$R_{\eta D}$ correspond to the graphs of their respective resolvents $\Jcal_{\eta B}$ and $\Jcal_{\eta D}$
through a linear transformation, $R_{\eta B}$ and
$R_{\eta D}$ are also polyhedral mappings. Since by Remark \ref{r:resolvents sv} the 
resolvents $\Jcal_{\eta B}$ and $\Jcal_{\eta D}$ are single-valued, the reflectors 
$R_{\eta B}$ and $R_{\eta D}$ are  
also single-valued.  Therefore $T=\frac{1}{2}( R_{\eta B} R_{\eta D}+I)$ is polyhedral as the composition of 
single-valued polyhedral mappings.
\end{proof}

\begin{thm}[local linear convergence II]\label{t:convergence-polyhedrality}
Let $J:U\to \Rbb\cup\{+\infty\}$ and $H:V\to \Rbb$ 
be proper, lsc, convex, piecewise linear-quadratic functions (see Definition \ref{d:plq}). 
Define the operator $\map{T}{V}{V}$ by \eqref{e:T_DR} with $\eta>0$ fixed and
$B$ and $D$ given by \eqref{e:BnD} where
$\map{A}{U}{V}$ is a  linear mapping.  
Suppose that there exists a solution to $0\in (B+D)(x)$, that   
 $\map{T}{W}{W}$ for $W$ some affine subspace of 
$V$ and that $\Fix T\cap W$ is an isolated point $\{\xbar\}$.   Then there is a 
neighborhood $\Ocal$ of $\xbar$ such that, for all 
starting points $(x^0,v^0, b^0)$ with 
$x^0\equiv b^0+\eta v^0\in \Ocal\cap W$ for $v^0\in D(b^0)$ so that $\Jcal_{\eta D}x^0=b^0$,  
the sequence $(x^k)_{k\in\Nbb}$ generated by \eqref{e:DRSalg2} 
converges linearly to $\xbar$ where  $\bbar\equiv \Jcal_{\eta D}\xbar$ is a solution 
to \eqref{e:Dp}.  The rate of linear convergence is bounded above by $\sqrt{1-\kappa}$, 
where $\kappa=c^{-2}>0$, for $c$ a constant of metric subregularity of $\Id-T$  at $\xbar$ for 
the neighborhood $\Ocal$.   
Moreover,  the sequence $\paren{b^k, v^k}_{k\in\Nbb}$ 
generated by Algorithm \ref{a:ADMM}  converges linearly to $\paren{\bbar,\vbar}$
 with $\vbar=\tfrac1\eta\paren{\xbar-\bbar}$, and the sequence 
$\paren{u^k}_{k\in\Nbb}$ defined by \eqref{ADMM1} of Algorithm \ref{a:ADMM} 
converges to a solution to \eqref{e:Pp}.
\end{thm}
\begin{proof}
   By Proposition \ref{t:polyhedral subdiff} the Douglas--Rachford operator $T$ is polyhedral and thus 
   the first statement follows from Corollary \ref{t:polyhedral convergence}.    The statement about 
   the sequences generated by Algorithm \ref{a:ADMM} follows as in Theorem \ref{t:convergence}. 
\end{proof}

\section{Error Bounds and Iterative Penalization}\label{meta-alg}
In this section, we study an iteratively regularized algorithmic scheme for solving
the problems of the form 
\[
 \min\set{J(u)}{u\in  U~\mbox{ and }f_j(Au)\leq \epsilon_j, ~j=1,2,\dots,M},   
\]
where $J: U \rightarrow \extre$ is proper lsc and 
convex, the mapping $A:U\rightarrow V$ is linear, for all $j$ the nonnegative-valued function
$\map{f_j}{V}{\Rp}$ is convex and smooth (at least at points that matter) and $\epsilon_j>0$. 
We refer to the inequality constraints as {\em structured constraints}.  It will
be convenient to introduce the following notation that will help to reduce clutter.  We collect the 
constraints into a vector-valued function so that we can write the problem as
\begin{equation*}\tag{\ensuremath{\Pcal}}\label{e:exact}
\cmin{J(u)}{u\in  U}{F_\epsilon(Au)\leq 0,}
\end{equation*}
where
\begin{equation}\label{e:Fepsilon}
 \map{F_\epsilon}{V}{\Rbb^M}\equiv v\mapsto \left(f_1(v)-\epsilon_1, f_2(v)-\epsilon_2, \dots, f_M(v)-\epsilon_M\right)^T.  
\end{equation}
Here the vector inequality is understood as holding element-wise.  

A common approach to solving problems of the type \eqref{e:exact} arising from inverse problems is
to apply {\em implicitly} the structured constraint by adding some (usually smooth) quantification 
of the constraint violation into the objective function:
\begin{equation*}\tag{\Pr}
\label{e:implicit}
\ucmin{J(u)+\rho\theta(F_\epsilon(Au))}{u\in U},    
\end{equation*}
where $\map{\theta}{\Rbb^M}{\extre}$ is a proper, lsc convex function and $\rho>0$.  
This places us in the context of the previous section since problem \eqref{e:implicit} is the 
specialization of \eqref{e:Pp} with $H(Au)=\rho\theta(F_\epsilon(Au))$.

As is often seen in the inverse problems literature, the constraint violation parameter 
$\epsilon_j=0$ ($j=1,\dots,M$), essentially penalizing divergence from the origin.
A prominent instance of this form of regularization is the squared norm:
$
   \theta(v)\equiv\|v\|^2. 
$
There are many efficient methods available for 
solving \eqref{e:implicit}. It is clear that for a certain value of $\rho$ 
the optimal solution to \eqref{e:implicit}, $u_\rho$, will satisfy
$f_j(Au_\rho)\leq \epsilonbar_j(\rho)$ with the {\em effective} error $\epsilonbar_j(\rho)$ depending 
on $\rho$. What is {\em not} true in general, however, is that the solution to  \eqref{e:implicit} corresponds
to the solution to \eqref{e:exact} for the constraint error $\epsilonbar(\rho)$.  Moreover, for 
our intended applications, $U$ is a finite dimensional Euclidean space with dimension $n$ and
the dimensionality of the constraints $M$ grows superlinearly as a function of $n$, so we would like to 
consolidate the constraints somehow while exploiting the phenomenon that, at the solution to \eqref{e:exact}
relatively few of the constraints are in fact tight or {\em active}.

We consider convex penalties that reduce the dimensionality of the constraint structure and 
have the property that $\theta(F_\epsilon(Au))=0$ if and only if $F_\epsilon(A u)\leq 0$.    
Of particular interest among penalties with this property are {\em exact penalties}, that is 
penalties $\theta$ with the property that solutions to \eqref{e:implicit} correspond to solutions to 
\eqref{e:exact} for all values of $\rho$ beyond a certain threshold $\rhobar$. 
For more background on exact penalization see, for example, 
\cite{Bert82, Fletcher85, Burke91, Mangasarian79, Mangasarian85, Conn2000}. 
We point also to Friedlander and Tseng \cite{FriedlanderTseng07} for a connection between exact penalization 
and what they call {\em exact regularization} as this fits well with our viewpoint that the structured constraints 
$F_\epsilon(Au)\leq 0$
constitute a regularization of the {\em model} with regularization parameter $\epsilon$.
This illustrates the distinction between {\em model-based} regularization, that is, regularization of the 
constraints motivated by external (eg. statistical) considerations, versus {\em numerical} regularization 
motivated solely on the grounds of enabling efficient (approximate) numerical solutions to \eqref{e:exact}.  

While it is nice to know that, with exact penalization, one can achieve an exact correspondence between the original constrained
optimization problem and the penalized problem, the whole point of relaxing the constraints is to 
reduce the computational burden of strictly enforcing the constraints.  As is often done in practice, one gradually strengthens the 
constraints, finding intermediate points that nearly solve the relaxed problem and using these as starting points 
for solving a more strictly penalized problem.    Together with Theorem \ref{t:dist to Sinf} below, the linear convergence rate 
established in Theorems \ref{t:convergence} and \ref{t:convergence-polyhedrality} of the previous 
section yield estimates on the distance
of intermediate points to the solution set of the relaxed problem as well as estimates on the distance 
to feasibility for the unrelaxed problem. 

\subsection{Structured Constraints and penalization}

Define 
\begin{equation}\label{e:Ccal}
\Ccal\equiv \set{u\in U}{F_\epsilon(Au)\leq 0}.
\end{equation} 
This is a closed convex set since the $f_j$ are lsc and convex. If there exists some $\alpha\in \Rbb$ 
such that $\Ccal \cap \lev_{\leq \alpha}J$ 
is nonempty and bounded 
then \eqref{e:exact} has a solution \cite[Theorem 11.9]{BauschkeCombettes11}.   
This will happen, for instance, if $\dom(J)\cap \Ccal\neq\emptyset$ and $J$ is {\em coercive} \cite[Proposition 11.12]{BauschkeCombettes11}, 
that is $J$ satisfies
\begin{equation}\label{e:coercive-j}
\lim_{\|u\|\toinf}J(u)=+\infty.
\end{equation}
Such assumptions are naturally satisfied in many applications.  Moreover, $\lev_{\leq \alphabar}J(u)$, 
the lower level-set of $J$ corresponding to the optimal value $\alphabar$ in \eqref{e:exact}, 
is convex and so the set of optimal solutions to \eqref{e:exact} is also convex.  
Define $J_\rho\equiv J+\rho \theta(F_\epsilon\circ A)$ for the convex, lsc function $\theta$ satisfying $\theta(w)\geq 0$ 
for all $w$ and  $\theta(w)=0$ if and only if $F_{\epsilon}(w)\leq 0$.  Then $J_\rho$ is convex, lsc and 
corresponds exactly to $J$ on the set $\Ccal$.  Otherwise $J_\rho$ increases pointwise to $+\infty$ at points 
outside $\Ccal$ as $\rho\to\infty$.   For $\paren{\rho_k}_{k\in\Nbb}$ with $\rho_k\to\infty$,  the sequence of functions 
$\paren{J_{\rho_k}}$ epi-converges (see \cite[Definition 7.1]{VA}) to $J+\iota_{\Ccal}$ as $k\to+\infty$ where $\iota_\Ccal$ 
is the indicator function of the set $\Ccal$. 
As we will allow {\em approximate} solution of problems \eqref{e:implicit} it will be helpful to recall the 
set of $\gamma$-minimizers: $\gamma-\argmin J_\rho\equiv\set{u}{J_\rho(u)\leq \inf J_\rho +\gamma}$. 
The relation between the solution sets to \eqref{e:exact} and \eqref{e:implicit} is detailed in the following, which is a direct 
application of \cite[Theorem 7.33]{VA}.
\begin{propn}\label{t:epi-convergence}
   Let $J: U \rightarrow \extre$,  $\map{F_\epsilon}{V}{\Rbb^M}$ and $\map{\theta}{\Rbb^M}{\Rbb}$ be proper, lsc and 
convex, and let $A:U\rightarrow V$ be linear.  Let $J$ be coercive with $\dom J\cap \Ccal\neq \emptyset$ for $\Ccal$ 
defined by \eqref{e:Ccal}.  Suppose further that $\theta(w)\geq 0$ and that 
$\theta(w)=0$ if and only if $F_{\epsilon}(w)\leq 0$.  Define $J_{\rho_k}\equiv J+\rho_k \theta(F_\epsilon\circ A)$
where $\rho_k\nearrow +\infty$ as $k\nearrow+\infty$.  Then $\inf J_{\rho_k}\to \inf J+\iota_\Ccal<+\infty$.  Moreover, 
for any sequence of errors $\gamma_k\searrow 0$ and corresponding points $u^k\in \gamma_k-\argmin J_{\rho_k}$, 
the sequence $\paren{u^k}_{k\in \Nbb}$ is bounded, and all its cluster points belong to $\argmin\{J+\iota_\Ccal\}$. 
\end{propn}
{\em Proof sketch.}
The property of the convex penalty $\theta$ that $\theta(w)\geq 0$ and 
$\theta(w)=0$ if and only if $F_{\epsilon}(w)\leq 0$ yields epi-convergence of $J_{\rho_k}$ to $J+\iota_\Ccal$.  Coercivity
of $J$ guarantees that $J_\rho$ is level bounded for all values of $\rho>0$.  These two properties, together with 
lower semicontinuity and the fact that $J$ and $J_\rho$ are proper, are all that is needed to prove the result.      
\hfill $\Box$

If the regularization were {\em exact}, then we would know that for 
all parameter values $\rho$ large enough, the solutions to \eqref{e:implicit} coincide with solutions to \eqref{e:exact}.  We 
return to this later.  

\subsection{Solution to the regularized Subproblem and error bounds}
We now turn our attention to solution of the problem \eqref{e:implicit} for a fixed value of $\rho_k$. 
The ADMM algorithm 
discussed in Section \ref{s:siadpmm} is useful for solving this problem in the sense
that it has an error bound under specific assumptions which gives a stopping rule.  
This is not unique to Algorithm \ref{a:ADMM}, but we focus 
on this method due to its prevalence in practice. 

Recall the exact problem \eqref{e:exact}:
\begin{equation*}\tag{\ensuremath{\Pcal}}
\cmin{J(u)}{u\in  U}{F_\epsilon(Au)\leq 0.}
\end{equation*}
It will be convenient to rewrite the penalized problem\footnote{Of course, the value of the problem is not the same, 
but the solutions are.} 
\eqref{e:implicit} as 
\begin{equation*}\tag{\Pr}
\label{penalized-swap}
 \ucmin{\frac{1}{\rho}J(u)+\theta(F_\epsilon(Au))}{u\in U}.
\end{equation*}
Consider also the {\em limiting} problem
\begin{equation*}\tag{\Pinf}
\label{e:swap}
 \ucmin{\theta(F_\epsilon(Au))}{u\in U}.
\end{equation*}
We view problem \eqref{penalized-swap} as the 
regularized version of  \eqref{e:swap} with $J$ as the regularizing
functional and $\frac{1}{\rho}$ as the regularization parameter. 
Denote the solution sets to these problems by 
\begin{eqnarray}
   S&\equiv& \argmin\set{J(u)}{u\in U, ~F_\epsilon(Au)\leq 0},\nonumber\\
   S_\rho&\equiv& \argmin\set{\frac{1}{\rho}J(u)+ \theta\paren{F_\epsilon(Au)}}{u\in U},\nonumber\\
   S_\infty&\equiv& \argmin\set{\theta\paren{F_\epsilon(Au)}}{u\in U}.\nonumber
\end{eqnarray}

If the penalization $\theta$ satisfies $\theta(F_\epsilon(Au))=0$
if and only if $F_\epsilon(A u)\leq 0$, then it is immediately 
clear that $S_\infty$ corresponds to the feasible set of problem \eqref{e:exact} hence $S\subset S_\infty$. 
What is more remarkable is that, if a Lagrange multiplier for \eqref{e:exact} exists, then $S_\rho=S$ for all 
$\rho$ large enough, that is, the penalty $\theta$ is {\em exact}.  
\begin{thm}[Theorem 4.2 of \cite{FriedlanderTseng07}]\label{t:exact penalty}
 Suppose that $S$ is nonempty and compact,  and that there exist
 Lagrange multipliers $\lambda$ for \eqref{e:exact}. Let the penalization $\theta$ in 
\eqref{e:implicit} be convex. Assume, moreover, that $\theta$ satisfies the condition $\theta(F_\epsilon(Au))=0$
if and only if $F_\epsilon(A u)\leq 0$.  Then the solution set to the penalized problem, $S_\rho$, 
 coincides with the solution set to the exact problem, $S$,
for all $\rho>\theta^\circ(\lambda)$ where 
 $\theta^\circ$ is the polar function of $\theta$ given by $\theta^\circ(\lambda)=
 \sup_{x\nleq 0}\frac{\lambda^Tx}{\theta(x)}$.
\end{thm}
It is easy to check whether a solution $u_\rho\in S_\rho$ is in fact feasible for \eqref{e:exact} (and 
hence also in $S$) by simply evaluating the value of $\theta\paren{F_q(Au_\rho)}$.  
More generally, one would check whether  the first order optimality conditions for \eqref{e:swap}
are satisfied at $u_\rho$, namely
\begin{equation}\label{e:foc swap}
 0\overset{?}{\in}\sd\theta\paren{F_\epsilon\circ A(\cdot)}\mbox{ at $u_\rho$}.
 \end{equation}
An explicit formula for the subdifferential in \eqref{e:foc swap} for image denoising and 
deconvolution is given in Section \ref{s:application} as this will be needed for computing 
Step \eqref{ADMM2} of Algorithm \ref{a:ADMM}. 

If, in addition, $S_\infty$ is {\em weakly sharp} (see Definition \ref{d:weak sharp}), then one can obtain 
an upper bound for the distance of solutions to \eqref{e:implicit} to {\em feasible} solutions to \eqref{e:exact}, even 
in the absence of Lagrange multipliers for \eqref{e:exact}.   

\begin{assumption}\label{a:assumption-FreidlanderTseng}$~$
 \begin{enumerate}[(i)]
  \item\label{a:assumption-FreidlanderTseng 1} The solution set $S_\infty$ of problem \eqref{e:swap} is nonempty.
  \item\label{a:assumption-FreidlanderTseng 2} $\lev_{\leq \alpha}J$ is bounded for each $\alpha \in \mathbb{R}$ and $\inf_{x\in U}>-\infty$. 
  \item\label{a:assumption-FreidlanderTseng 3} The solution set $S_\infty$ of \eqref{e:swap} is weakly sharp of order $\nu\geq 1$.
 \end{enumerate}
\end{assumption}
\begin{thm}\label{t:dist to Sinf}
Suppose Assumption \ref{a:assumption-FreidlanderTseng}\eqref{a:assumption-FreidlanderTseng 1}-\eqref{a:assumption-FreidlanderTseng 2} hold. 
\begin{enumerate}[(i)]
   \item\label{t:dist to Sinf i} For any $\rhobar>0$, $\bigcup_{\rho\geq\rhobar}S_\rho$ is bounded.  
   \item\label{t:dist to Sinf ii} If, in addition,  Assumption \ref{a:assumption-FreidlanderTseng}\eqref{a:assumption-FreidlanderTseng 3} 
holds with modulus
  of sharpness $\nu$, then for any $\rhobar>0$ there exists $\tau>0$ such that
  \begin{equation}\label{e:dist to Sinf}
   dist(u_\rho, S_\infty)^{\nu-1}\leq \frac{\tau}{\rho}, ~~\forall u_\rho\in S_\rho, ~~\rho\geq \rhobar.
  \end{equation}
   \item\label{t:dist to Sinf iii} If, in addition,  Assumption \ref{a:assumption-FreidlanderTseng}\eqref{a:assumption-FreidlanderTseng 3} holds
and the penalization $\theta$ is exact, then for all $\rho$ large enough, $u_\rho\in S$ and $dist(u_\rho, S_\infty)=dist(u_\rho, S)=0$.
\end{enumerate}
\end{thm}
\begin{proof}
\eqref{t:dist to Sinf i} and \eqref{t:dist to Sinf ii}.   
Under the assumption \ref{a:assumption-FreidlanderTseng}, Theorem 5.1 in \cite{FriedlanderTseng07} directly applies to yield
the result. \hfill$\triangle$

\eqref{t:dist to Sinf iii}.
If the penalization $\theta$ is exact, then  $\theta(F_\epsilon(Au))=0$ if and only if $F_\epsilon(A u)\leq 0$, hence 
$S=S_\rho$ for all $\rho$ large enough, and $S_\infty$ corresponds exactly to the feasible set in \eqref{e:exact}.
\end{proof}

\begin{remark} The error bound \eqref{e:dist to Sinf} holds independent of the existence of Lagrange multipliers for \eqref{e:exact}, 
hence, for exact penalization under Assumption \ref{a:assumption-FreidlanderTseng},  
Theorem \ref{t:dist to Sinf} yields an upper bound on the distance of solutions to \eqref{e:implicit} to feasible points for \eqref{e:exact}.    
\end{remark}

\section{Application:  image deconvolution and denoising with statistical multiscale analysis}
\label{s:application}

We specialize the above results to the application of optimization with statistical multiscale 
side constraints.  All of the examples considered in this section satisfy the requirements
of Theorem \ref{t:convergence-polyhedrality}, and thus for each fixed value of the 
penalty parameter $\rho$ local linear convergence to a solution of \eqref{e:implicit} is guaranteed. 
Moreover, the penalty function $\theta$ that we use is {\em exact} and hence by 
Theorem \ref{t:exact penalty}, for $\rho$ large enough, the computed solution to \eqref{e:implicit}
is also a solution to \eqref{e:exact}.  What is not known {\em a priori} is what value of $\rho$ yields
the correspondence.  Moreover, since the whole point of the relaxation \eqref{e:implicit} is 
to remove the burden of satisfying the constraints, we approach a solution to \eqref{e:exact}
via a sequence of solutions to \eqref{e:implicit} for progressively larger values of $\rho$.
This is described precisely in the following sequentially penalized algorithm. 
\begin{center}
\fbox{%
        \addtolength{\linewidth}{-2\fboxsep}%
        \addtolength{\linewidth}{-2\fboxrule}%
    \begin{minipage}{\linewidth}%
	\begin{algorithm}[{\bf Exactly Penalized Sequential ADMM}] \label{a:DADMM}$~$\\
{\bf Initialization}. Given an image $y$,  
a sequence of error tolerances $(\gamma_k)_{k \in \mathbb{N}}$ with $0\leq \gamma_k \to 0$ 
Choose parameters: $\beta>1$ and the penalty parameter $\eta\in (0,2)$.
Initialize $ k=i=0$, $b^{(0, 0)}=0, v^0=y$, $u^{(0, 0)}=A^{T}y$,
and  compute 
$u^{(0, 1)}=\argmin_{u}\left\{J(u)+\langle b^{(0, 0)}, Au\rangle +
\tfrac{\eta}{2}\|Au-v^{(0, 0)}\|^2 +\tfrac{1}{2}\|u-u^{(0, 0)}\|^2\right\}$.\\

{\bf For $k=0,1,2,\dots$}
\begin{itemize}
\item {\bf While $\|u^{(k, i+1)}-u^{(k, i)}\|>\gamma_k$}
\begin{itemize}
\item Compute $(v^{(k, i+1)}, b^{(k, i+1)})$ via Algorithm \ref{a:ADMM} steps \eqref{ADMM2}-\eqref{ADMM3}
with $H\equiv \rho_k\theta\paren{F_q\paren{\cdot}}$ for the {\em exact} penalty $\theta$ and structured constraints 
$F_q$.
\item Increment $i=i+1$ and calculate $u^{(k, i+1)}$ via Algorithm \ref{a:ADMM} step \eqref{ADMM1}.
\end{itemize}
\item {\bf Update/reset:}  Set $u^{(k+1, 1)}:=u^{(k, i+1)}$ and $\rho_{k+1}=\beta\rho_{k}$.  Set $k=k+1$ and  $i=0$.
If $\theta\paren{F_q\paren{u^{(k, 1)}}}=0$, set $\gamma_k=0$.\\
\end{itemize}
\end{algorithm}
\end{minipage}
}
\end{center}
The outer iteration, indexed by $k$, consists of numerical approximations to solutions of 
\eqref{e:implicit} for the penalty parameter $\rho_k$. The the inner iteration proceeds with 
the current value of $\rho_k$ until
the step size between successive iterates $u^{k,i+1}$ and $u^{k,i}$ drops in a linear 
fashion below a given tolerance $\gamma_k$.  From Theorem \ref{t:convergence-polyhedrality}
one can then obtain a posteriori estimates on the distance of the iterate $u^{k+1,j}$
to the true solution. Then $\rho_k$ is increased by a constant factor.  
Since, for this model the penalization $\theta$ is exact, once the constraints appear to be satisfied 
(as determined by monitoring the value of $\theta\paren{F_q\paren{v^k}}$), 
it is reasonable to conclude that the correspondence between problems \eqref{e:implicit} and 
\eqref{e:exact} holds, and 
the penalty $\rho_k$ no longer 
needs to be updated; the inner loop of the algorithm then can be run to the desired 
accuracy.  As indicated in Figures \ref{Fig:s-data}(b) and \ref{Fig:performance}, 
the constraints appear to be satisfied when the penalty term $\rho_k\theta\paren{F_q\paren{v^k}}$ 
 (green plot) drops  suddenly to machine precision.

The application problem involves image deconvolution and denoising with 
statistical multiscale estimation as presented in \cite{Aspelmeier2015, FrickMarnitzMunk12, FrickMarnitzMunk13}.  
We are well aware that there are many ways to model such problems that permit much less computationally 
intensive numerical solutions than the technique we present here.  Our interest in multiresolution 
deconvolution/denoising is two-fold:  first, it is one of the few techniques
available that has the potential to yield quantitative (i.e. statistical) guarantees for the recovered images, and secondly, 
it is an important instance of convex optimization problems where the number of constraints grows superlinearly 
as a function of the number of unknowns.  Our numerical demonstration addresses the first issue of quantitative 
image denoising:  if the numerics do not permit estimates for the distance to the model solution, then the quantitative 
assurances of the model are irrelevant.  Unlike the numerical approach proposed in 
\cite{FrickMarnitzMunk12, FrickMarnitzMunk13},
the numerical approach we present here permits error bounds to within
machine accuracy of our numerical solution to the true model solution.  

Following the approach proposed in \cite{FrickMarnitzMunk12} we quantify the difference between an estimate 
$v=Au$ and the data $y$
via the maximum absolute value of all  weighted inner products of the residual function $\triangle(\cdot; y):\Rn\to \Rn$:
\begin{equation}\label{e:residual j}
f_j(v)\equiv \left|\left\langle\omega_{j}, \triangle(v; y) \right\rangle\right|, \quad j\in \{1,2,\dots,M\}.
\end{equation}
The residual function used in \cite{FrickMarnitzMunk12} $\triangle$ is simply $v-y$.  
The weights $\omega_{j}$ are scaled window functions so that the set $\mathcal I\subset \{1,2,\dots,M\}$ is 
the index set corresponding to all collections of these subsets of the image.    
The statistical multiscale analysis requires that, on each window, 
\begin{equation}
\max_{j\in \mathcal I}\{f_j(v)\} \leq q.
\label{e:smre2}
\end{equation}
The same error $q$ is specified at all scales.  Hence $F_\epsilon$ in \eqref{e:Fepsilon}
specializes to 
\begin{equation}\label{e:Fq}
\map{F_q}{\Rn}{\Rbb^{M+1}}\equiv v\mapsto \left(f_1(v)-q, f_2(v)-q, \dots, f_M(v)-q,0\right)^T,
\end{equation}
for $\map{f_j}{\Rn}{\Rbb}$ defined by \eqref{e:residual j} ($j=1,\dots,M$) and 
\begin{equation}\label{e:theta}
 \map{\theta}{\Rbb^{M+1}}{\Rbb}:\theta(w)\equiv \max\{w_1,w_2,\dots,w_{M+1}\}.
\end{equation}  
(Here we are expanding the original $F_\epsilon$ by the constant function $f_{M+1}(v)\equiv 0$.) 
The max function is a standard tool in exact penalization methods \cite{Burke91, Conn2000}
and falls naturally into the context of piecewise linear-quadratic functions.  

Algorithm \ref{a:DADMM} does not specify how 
the iterates $u^{(k,j)}$ and $v^{(k,j)}$ are calculated.  The linear convergence of the 
inner iterations predicted in Theorem \ref{t:convergence-polyhedrality}, from which error bounds can be 
determined,  as well as the numerical convergence of the outer iterates to problem \eqref{e:exact}
is discussed next.

\subsection{Prox-evalutation}
Computation of $u^{(k, i+1)}$ and $v^{(k, i+1)}$ in Algorithm  \ref{a:DADMM} 
involves minimizing the sum of a convex quadratic function and (in general) a convex, nonsmooth, 
piecewise linear-quadratic function. This can be solved via any number of techniques ranging from 
first order methods like FISTA \cite{BeckTeboulle09} to higher-order nonlinear optimization methods like 
quasi-Newton methods studied in 
\cite{LewisOverton13}.  In order to take advantage of the relative sparsity of the active constraints, 
we propose the following (exact) algorithm.
\begin{center}
\fbox{%
        \addtolength{\linewidth}{-2\fboxsep}%
        \addtolength{\linewidth}{-2\fboxrule}%
    \begin{minipage}{\linewidth}%
	\begin{algorithm}[{\bf Steepest Subdifferential Descent}] \label{a:SP}$~$\\
{\bf Initialization}. Given $b$, $u$, the constant $\eta>0$ and an initial point $v^0$,  compute the residual $r^0\equiv b + \eta Au-\eta v^0$
and the projected residual $z^0\equiv P_{\sd(\theta(F_q(v^0)))}(r^0)$ for $\sd(\theta(F_q(v^0)))$ given by \eqref{e:sd theta2}. 
\\

{\bf For $l=0,1,2,\dots$}
\begin{itemize}
\item {\bf If $z^l = r^l$}
\begin{itemize}
\item set $\vbar=v^l$ and STOP;
\end{itemize}
\item {\bf else}
\begin{itemize}
\item set $v^{l+1}=v^l+\lambda_l \paren{z^l-r^l}$ where $\lambda_l>0$ is the largest constant $\lambda$ such that 
$\theta\paren{F_q\paren{v^l+\lambda \paren{z^l-r^l}}}=f_i\paren{v^l+\lambda \paren{z^l-r^l}}-q$ for $i\in I(v^l)$
with 
\begin{equation}\label{e:active index}
  I(v)\equiv\set{j}{f_j(v)-q=\theta(F_q(v))}; 
\end{equation}
\item compute $r^{l+1}\equiv b + \eta Au-\eta v^{l+1}$ and 
the projected residual 
\begin{equation}\label{e:proj resid}
z^{l+1}\equiv P_{\sd(\theta(F_q(v^{l+1})))}(r^{l+1});   
\end{equation}
\item increment $l=l+1$.
\end{itemize}
\end{itemize}
\end{algorithm}
\end{minipage}
}
\end{center}
Algorithm \ref{a:SP} is an {\em active set method} and the set $I(v)$ defined by 
\eqref{e:active index} is the set of {\em active indexes at $v$}.  Another helpful 
interpretation is as a steepest subgradient descent method for solving  
\begin{equation}\label{e:active set eq}
  \argmin_{v}\klam{G(v)\equiv \theta\paren{F_q(v)}-\langle b, v\rangle+\tfrac{\eta}{2}\|Au-v\|^2}.
\end{equation}
The steepest descent step is 
\[
v^{l+1} = v^l+\lambda_ld^l, 
\]
for $d^l\equiv P_{\sd G(v)}(0) = -r^l + z^l$ with 
$z^l\equiv P_{\sd\theta\paren{F_q\paren{v^l}}}\paren{r^l}$
and $r^l = b+\eta\paren{Au-v^l}$. 
The choice of the step length $\lambda_l$ ensures that, at each step $l$, the active set is 
growing;  specifically,  
\[
I\paren{v^l}\subset I\paren{v^l+\lambda_l d^l}.
\]
At termination, the subdifferential $\sd\theta\paren{F_q\paren{v^l}}$ is large enough that 
it contains the residual $r^l$.
The terminal point of Algorithm \ref{a:SP}, $\vbar$, is a point in \eqref{e:active set eq} 
since it satisfies the first-order optimality conditions:
\begin{equation}\label{e:foc ADMM2}
   0=\zbar - b - \eta\paren{Au-\vbar}\in \sd\theta\paren{F_q(\vbar)}-b-\eta\paren{Au-\vbar}=\sd G(\vbar),
\end{equation}
where $\zbar= P_{\paren{\sd\theta\paren{F_q(\vbar)}}}\paren{b+\eta\paren{A u-\vbar}}$.  
Replacing $u$ and $b$ with $u^{(k, i+1)}$ and $b^{(k, i)}$ respectively 
yields the update for $v^{(k,i)}$ in Algorithm \ref{a:DADMM}.

 The expression for the subdiffferential $\sd\theta\paren{F_q}$ is particularly simple in this case. 
Note that $I(v)\neq \emptyset$ for all $v$.  
Applying the (convex) calculus of subdifferentials to the objective 
$\theta(F_\epsilon(v))$, as permitted by the regularity of $\theta$ and $F$
(see, for instance \cite[Section 2.3]{Clarke83}), yields 
\begin{equation}\label{e:sd theta}
 \sd\theta\paren{F_q(v)} = 
	\hull\set{\nabla f_j(v)}{ j\in I(v)},
\end{equation}
where $\hull$ denotes the convex hull of a set of points.  This, of course, assumes that $f_j$ is 
differentiable at $v$ for those $j\in I(v)$.  Inspection of \eqref{e:residual j} shows that this is not the 
case in general, in particular at points $v^*$ where $f_j(v^*)=0$.  However, such points will never be in the 
active set $I(v^*)$ since $f(v^*)-q<0\leq \theta\paren{F_q\paren{v^*}}$ for all $q>0$, so we can safely 
apply formula \eqref{e:sd theta} without further ado.  This yields the following specialization for 
$f_j(v)=\left|\ip{w_j}{v-y}\right|$ given by \eqref{e:residual j}:
  \begin{eqnarray}\label{e:sd theta2}
\sd\theta\paren{F_q\paren{v}} &=& \hull\set{\nabla f_j(v)}{ j\in I(v)} \\
&=&	\begin{cases}
	  \hull\left\{\set{\sign\left(\ip{w_j}{v-y}\right)w_j}{j\in I(v)\setminus \{M+1\}},~0\right\}&\theta(F_q(v))\leq 0\\
	\hull\set{\sign\left(\ip{w_j}{v-y}\right)w_j}{j\in I(v)}&\theta(F_q(v))>0.	  
	\end{cases}\nonumber
\end{eqnarray}

\subsection{Synthetic data}
Fig.~\ref{Fig:s-data} shows a set of synthetic
exact data $u^*\in \mathbb R^n$ (shown in blue) and corresponding noisy data $y\in
\mathbb R^n$ (shown in green) with $n=512$ data points, as well as the 
reconstructed/denoised signal $\ubar\in\mathbb R^n$ (shown in red).  
In this example we consider only denoising, that is, the imaging operator $A$ is the identity so $v=u$.
The noisy data $y$ was generated by adding i.i.d.\ Gaussian random noise with standard deviation $\sigma=0.05$ to each
original data point of $u^*$.
\begin{figure}
(a) \includegraphics[height=.4\textheight]{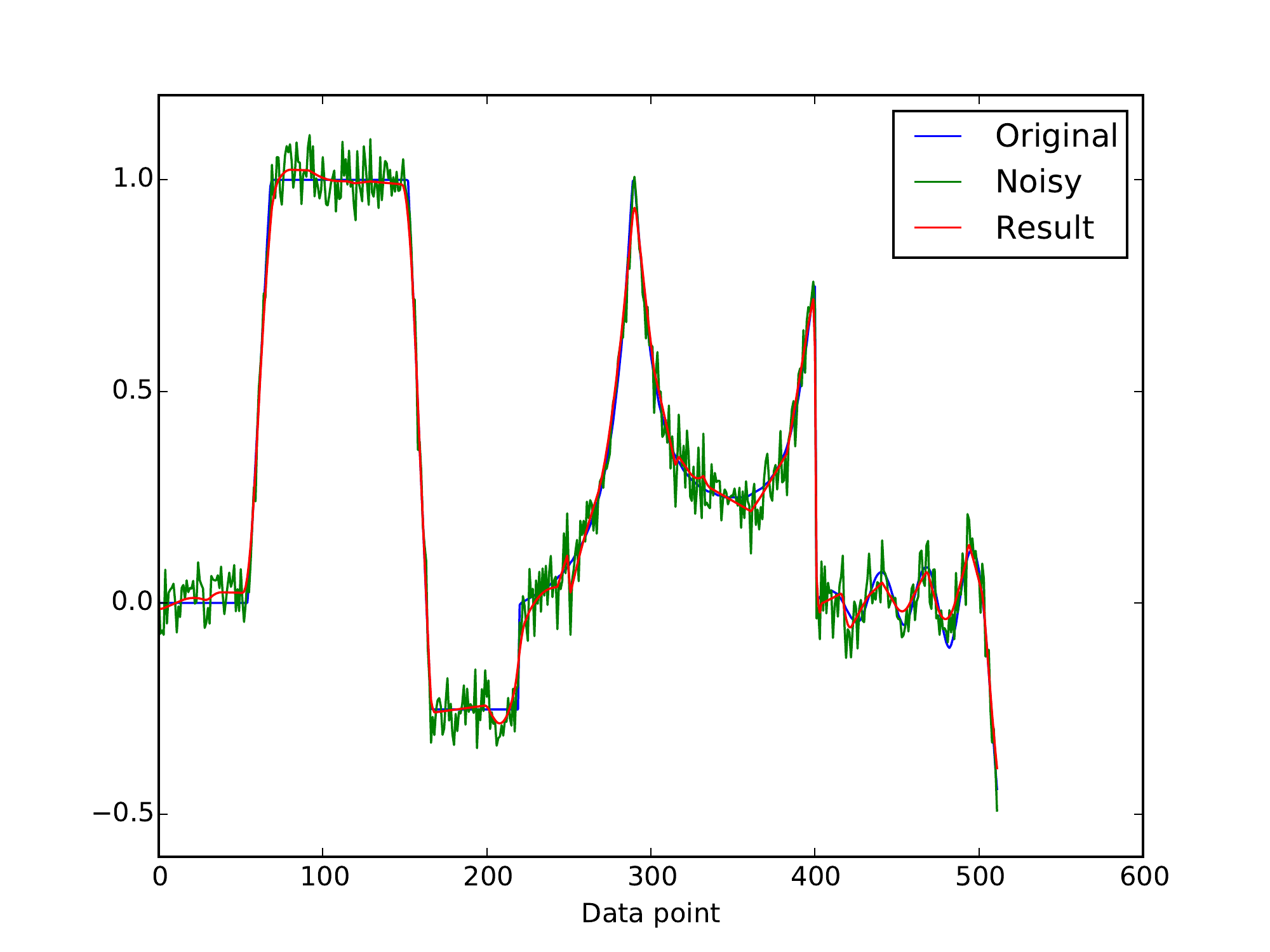}\\
(b) \includegraphics[height=.2\textheight]{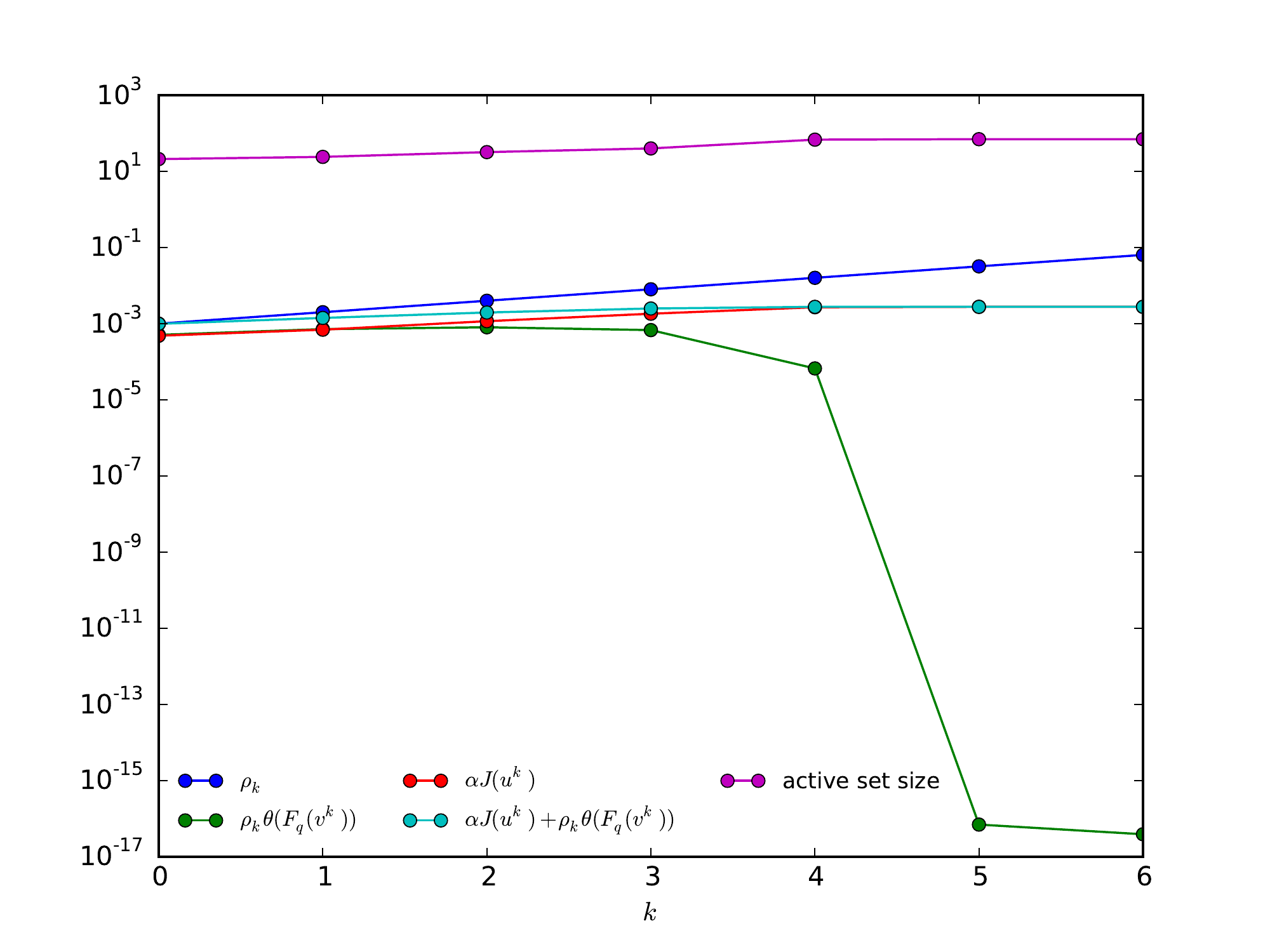}
(c) \includegraphics[height=.2\textheight]{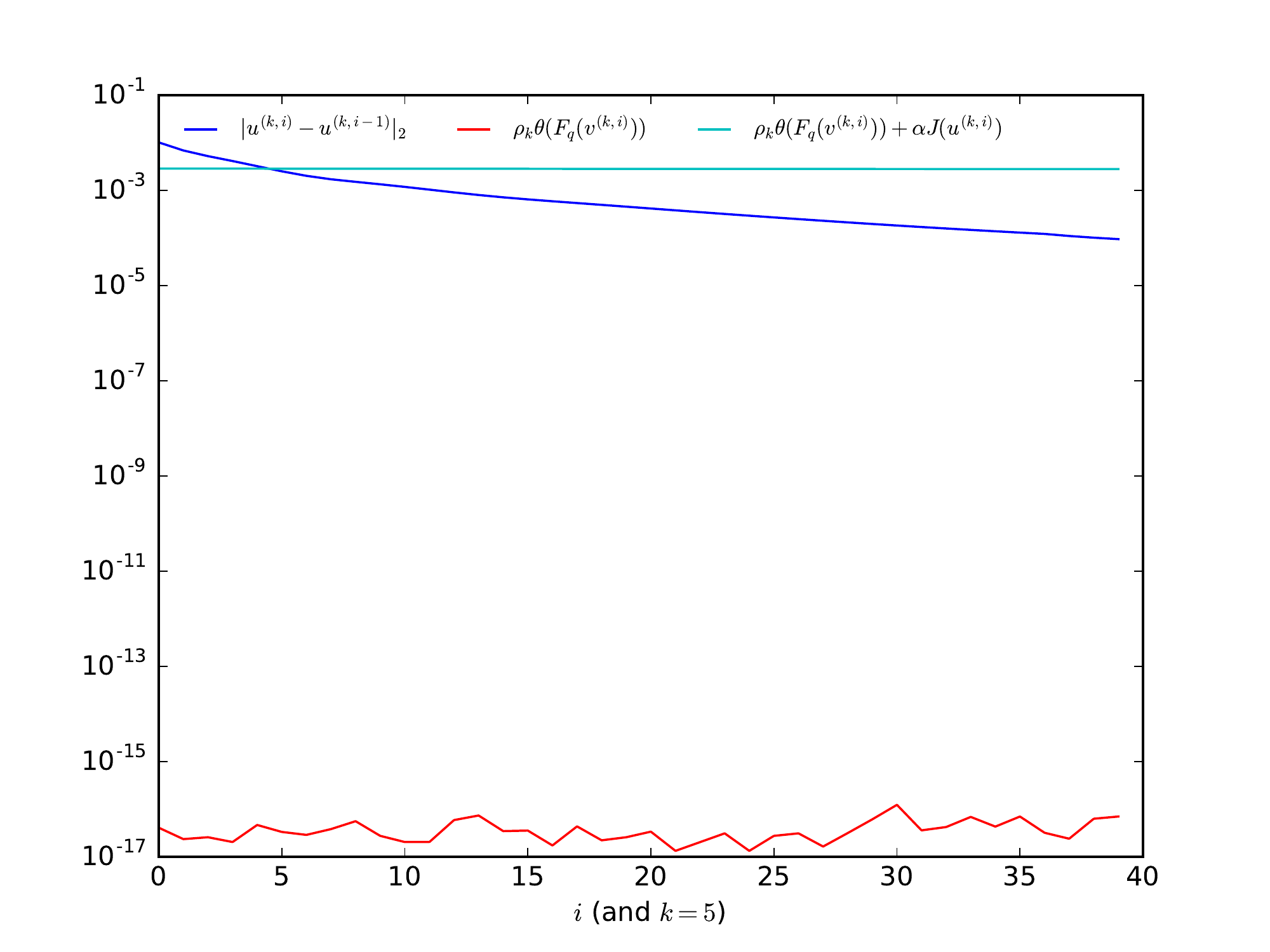}
\caption{(a) Original, noisy and reconstructed data for a one-dimensional
  denoising problem. (b) Outer iterates $k$ of Algorithm \ref{a:DADMM} showing solutions, 
  the constraint violation, the active set size and the objective value for 
  the penalized problem \eqref{e:implicit} for successively larger values 
  of the penalty parameter $\rho$.
(c) Inner iterates of Algorithm \ref{a:DADMM} with $\rho_{5}=.032$:  step sizes, 
constraint violation, objective value and gap between the primary, domain-space variables $u^{(k, i)}$.  }
\label{Fig:s-data}
\end{figure}
In our specialization of 
problem \eqref{e:exact} we use the total variation penalty
\begin{equation}
J(u) \equiv a ||\nabla u||_2^2,
\label{e:smre1-s}
\end{equation}
where $\nabla$ is the (discrete) gradient operator.  The structured constraints are given by 
\eqref{e:smre2}.
The weights $w_{j}\in\mathbb R^n$ are scaled window functions of all
intervals of lengths between $1$ and $20$ pixels, and $\mathcal I\subset \{1,2,\dots,M\}$ is 
the index set corresponding to all collections of successive pixels in $\{1,2,\dots,n\}$ of cardinality -- or length -- 
from $1$ to $20$.   The same error $q$ is specified at all scales.  

For a signal length $n=512$ with 
interval lengths from $1$ to $20$ the number of windows is $M=10050$.  The constant $\alpha$ is, strictly speaking,
redundant but was introduced as an additional means to balance the
contributions of the individual terms to make the most of limited numerical accuracy (double precision). 
We chose $\alpha=0.01$.  The constant $q$ was taken to be $2\sigma$.

Figure \ref{Fig:s-data}(a) shows very good correspondence of the reconstructed signal to the 
original.  The multi-resolution constraint prevents the usual ``blocky'' artifacts common to 
image denoising with TV-regularization.  The eventual (starting from around iteration $15$) 
linear convergence of the algorithm can be seen in Figure \ref{Fig:s-data}(c).  
Under the assumption that the latter iterates are indeed in the region of local linear 
convergence, the observed convergence rate is $c=0.9245$, which yields an 
a posteriori upper bound on the distance of the $39$th iterate to the true solution: 
$\|u^{39}-u^*\|\leq \frac{c}{1-c}\|u^{38}-u^{39}\|=0.001244$.  Since the signal length is $512$, 
this amounts to $5$ digits of accuracy in the pointwise value of the signal.  

\subsection{Laboratory data}
For our main demonstration, we are presented with an image $y\in \Rn$ (Figure \ref{Fig:edata}(a)) generated 
from a Stimulated Emission Depletion (STED) microscopy experiment 
\cite{Hell1994, Hell2000}
conducted at the 
Laser-Laboratorium G\"ottingen examining tubulin, represented as
the ``object'' $u\in\Rm$.  The imaging model is simple linear convolution, $Au\approx y$ where $A$ is 
a convolution matrix with a nonsymmetric experimentally measured point-spread function ($290\mathrm{nm}^2$).  
The measurement $y$ is {\em noisy} or otherwise inexact, and 
thus an exact solution $Au=y$ is not desirable.   Although the noise in such images is usually  
modeled by Poisson noise, a Gaussian noise model with constant variance suffices as the photon counts 
are of the order of $100$ per pixel and do not vary significantly across the image.
Figure~\ref{Fig:edata}(b) shows a close-up which we used as the noisy data $y\in
\mathbb R^2$ with $n=64\times 64$ data points.  We calculate the numerically reconstructed 
tubulin density $\ubar$ shown in Figure \ref{Fig:soln}(a) via Algorithm \ref{a:DADMM} 
for the problem \eqref{e:implicit} with the qualitative objective
\begin{equation}
J(u) \equiv \alpha || u||^2.
\label{e:smre1}
\end{equation}
   
\begin{figure}
\begin{tabular}{llll}
(a)&\includegraphics[height=0.4\textwidth]{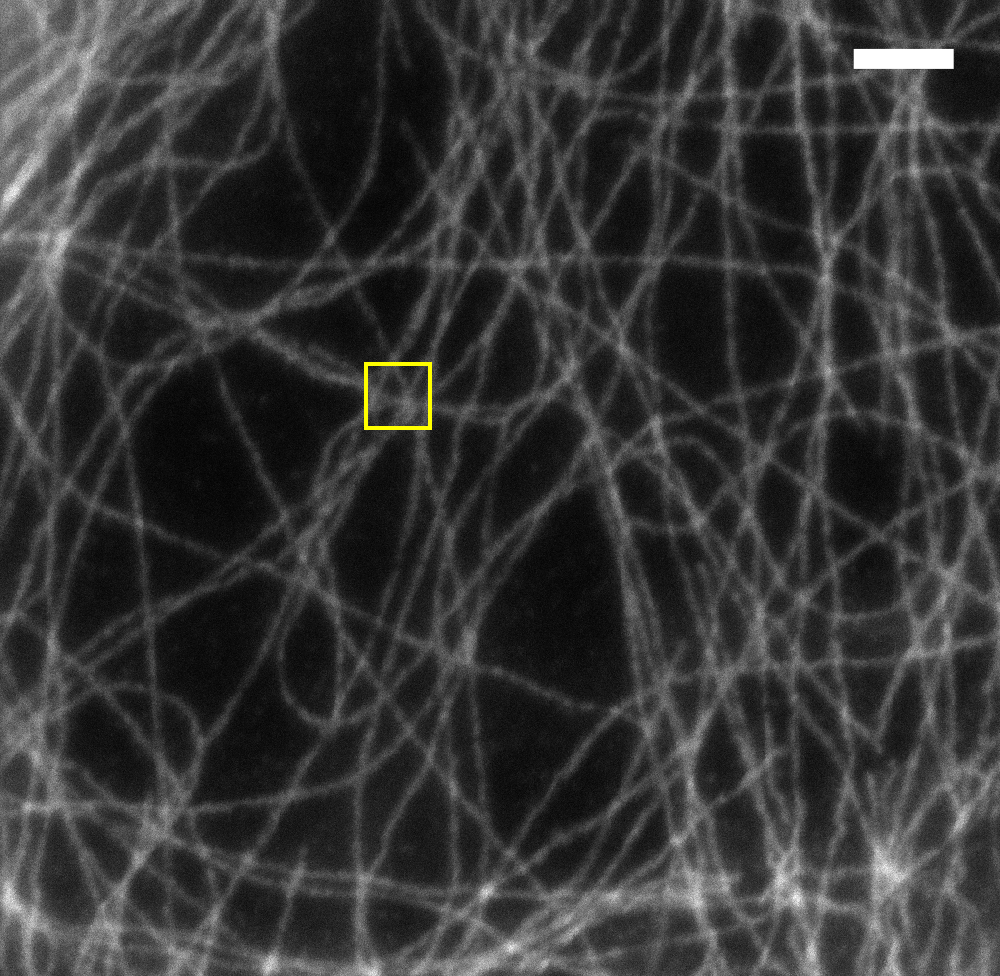}&
(b)&\includegraphics[height=0.4\textwidth]{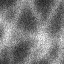}\\
\end{tabular}
\caption{(a) Original data (STED image of Tubulin), (b) an enlargement of the indicated box to be processed.
The length of scale bar in (a) is $1\mu$m, the size of the reconstruction window (b) is $640\times 640\,\mathrm{nm}^2$.}
\label{Fig:edata}
\end{figure}

\begin{figure}
\begin{tabular}{llll}
(a)&\includegraphics[height=0.4\textwidth]{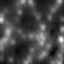}&
(b)&\includegraphics[height=0.4\textwidth]{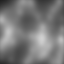}
\end{tabular}
\caption{(a) Numerical reconstruction via Algorithm \ref{a:DADMM} from the imaging data shown 
in Figure \ref{Fig:edata} for $\rho=4096$.
(b) The reconstruction convolved with the measured PSF.    
At each resolution used for the reconstruction, the sum of the pixel values in 
(b) lie within a confidence interval of $3\sigma$ of those in Figure \ref{Fig:edata}(b).}
\label{Fig:soln}
\end{figure}
\begin{figure}
\begin{tabular}{ll}
(a)&\includegraphics[height=0.4\textheight]{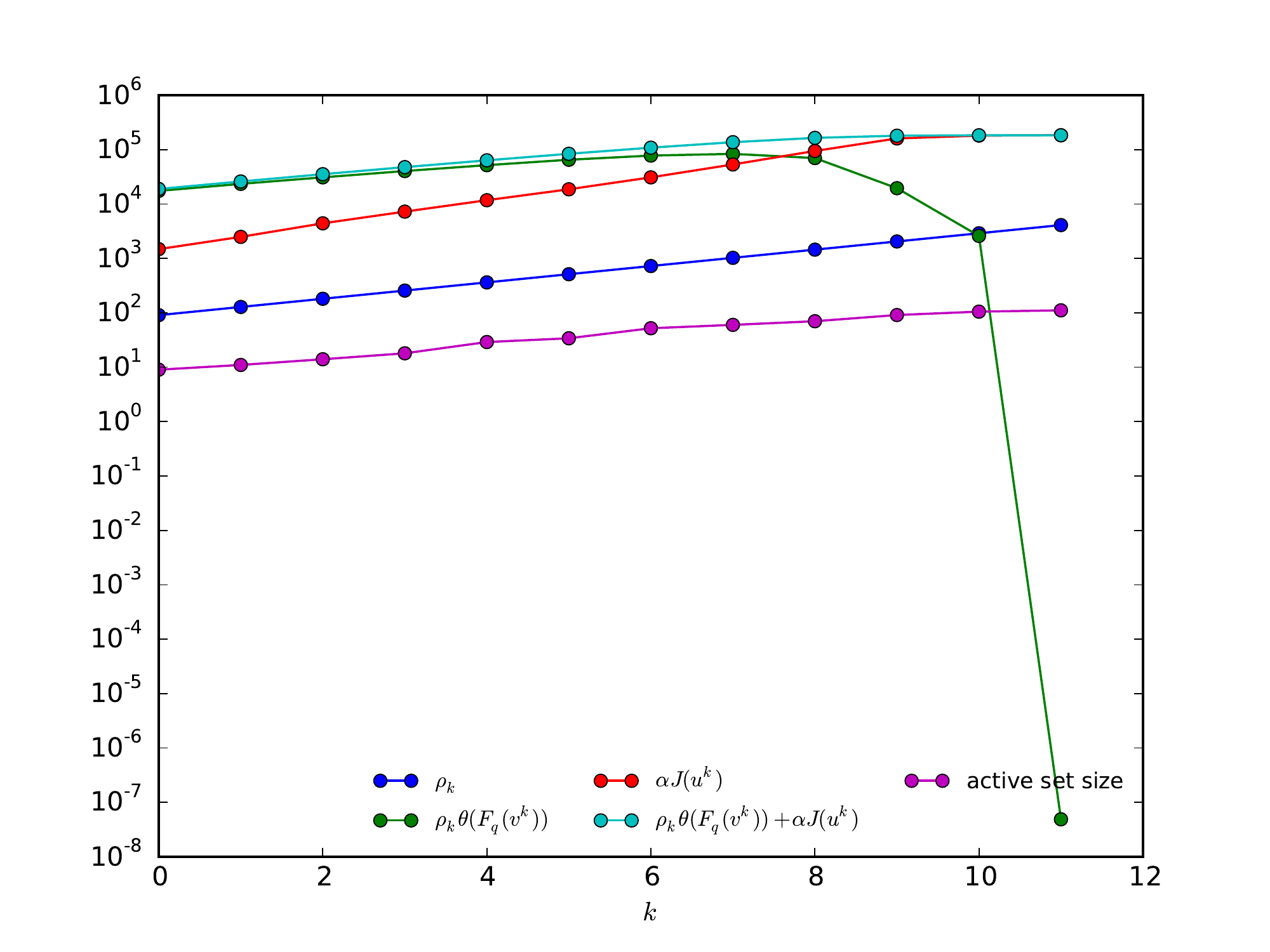}\\
(b)&\includegraphics[height=0.4\textheight]{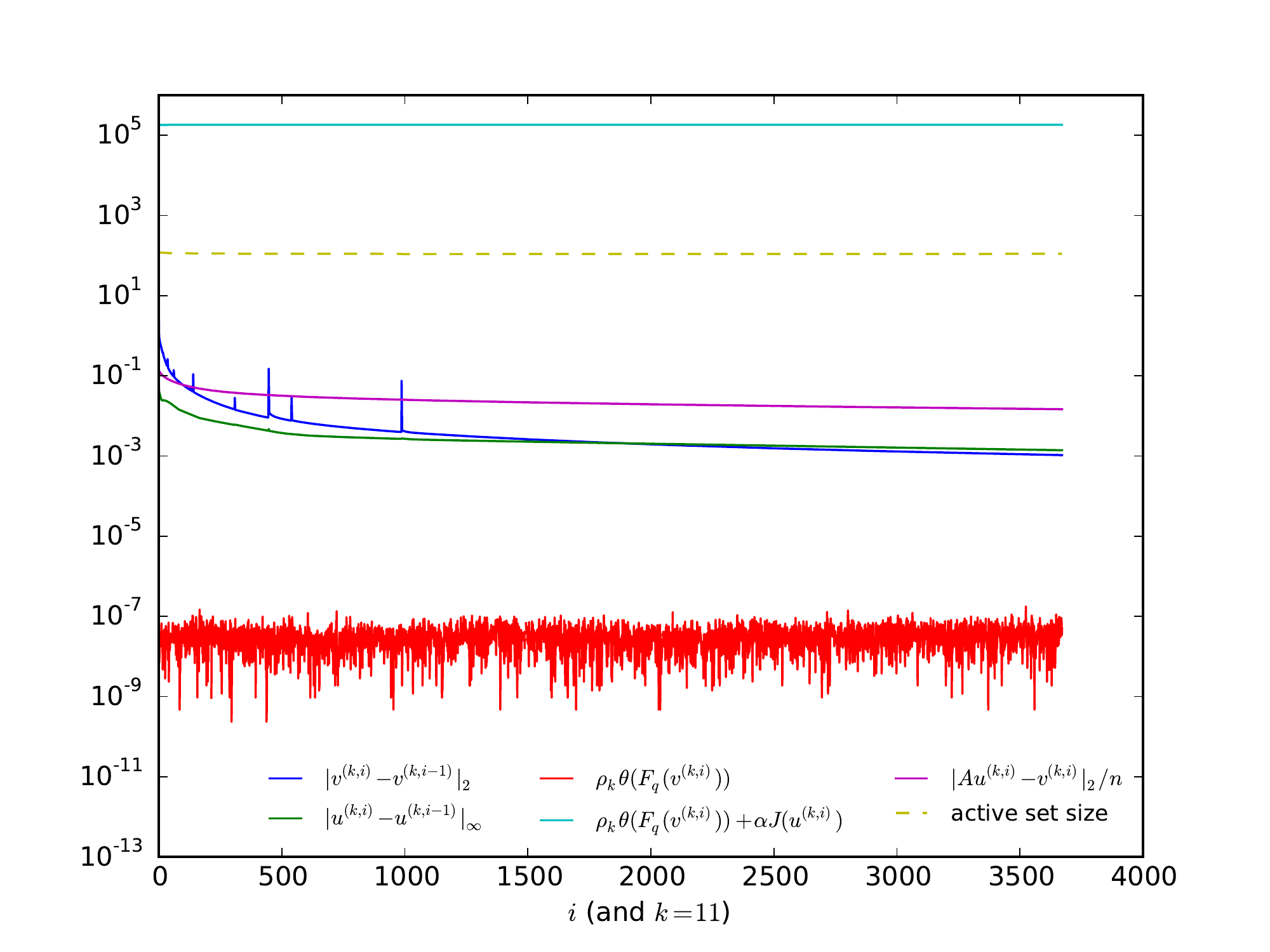}
\end{tabular}
\caption{(a) Outer iterates $k$ of Algorithm \ref{a:DADMM} showing solutions, 
  constraint violation, the value of the regularizer, the objective value for 
  the penalized problem \eqref{e:implicit} and the active set size for successively larger values 
  of the penalty parameter $\rho$.
(b)  Inner iterates of Algorithm \ref{a:DADMM} with $\rho_{11}=4096$:  step sizes, 
constraint violation, objective value and gap between the auxiliary, image-space  
variables $v^{(k, i)}$ and the primary, domain-space variables $u^{(k, i)}$.  }
\label{Fig:performance}
\end{figure}

For the image size $n=64\times 64$ with the window system of
squares of lengths $1$ and $2$, the number of windows is $M=8065$.  
The constant $\alpha$ in \eqref{e:smre1} is, strictly speaking,
redundant but was introduced as an additional means to balance the
contributions of the individual terms to make the most of limited numerical accuracy (double precision). 
We chose $\alpha=0.01$.  The constant $q$ was chosen so that the model solution would be 
no more than $3$ standard deviations from the noisy data on each interval of each scale.

We emphasize that, since this is experimental data, there is no ``truth'' for comparison - the constraint, 
together with the error bounds on the numerical solution to the model solution provide statistical 
guarantees on the numerical reconstruction \cite{FrickMarnitzMunk12}.   The numerical ``image'' generated from 
the reconstructed tubulin density, $\ubar$, is given by $\vbar=A\ubar$ and is shown in Figure \ref{Fig:soln}(b);  this 
figure is a denoised version of the measured data shown in Figure \ref{Fig:edata}(b). 

In Figure~\ref{Fig:performance}(a) a sample run of the algorithm shows a succession of 
outer iterations. The inner iteration is shown in Figure \ref{Fig:performance}(b) with 
the value of $\rho_{11}=4096$ for which the 
constraints are exactly satisfied (to within machine precision), indicating the correspondence
of the computed solution of problem \eqref{e:implicit} to a solution to the exact model problem \eqref{e:exact}.
The eventual (starting from around iteration $1500$) 
linear convergence of the algorithm can be seen in Figure \ref{Fig:performance}(c).  
Under the assumption that the latter iterates are indeed in the region of local linear convergence, the observed 
convergence rate is $c=0.9997$, which yields an 
a posteriori upper estimate of the pixelwise error of about $8.9062e^{-4}$, or $3$ digits of accuracy at each pixel.  
 
\section{Concluding remarks}
We have focused our attention on the ADMM algorithm due partly to its prevalence in practice, and 
partly its amenability to our theoretical techniques.  The parameter $\eta$ in Algorithm \ref{a:ADMM}
was left constant.  How to choose this parameter in the context of minimization is a perplexing question 
and worthy of further study.  Our theoretical framework can also be adapted to Krasnoselski-Mann relaxations of 
the Douglas--Rachford algorithm.  Statements about this will appear in work underway 
studying more generally {\em averaged mappings}.

The statistical interpretation of the reconstruction in Figure \ref{Fig:soln}(b) 
as described in \cite{FrickMarnitzMunk12, FrickMarnitzMunk13} opens the door to a quantitative 
approach to image processing, but this is only valid when 
one can estimate the distance of the numerical approximation to the exact solution to the 
underlying model optimization problem \eqref{e:exact}.  
Determining quantitative estimates for how 
close the numerical solution shown in Figure \ref{Fig:soln}(a) is to an exact solution to 
problem \eqref{e:exact} under the assumption of exact evaluation of the associated prox operators 
of has been the topic of our study.  

What is needed and largely missing in the current treatment of 
algorithms in the literature 
is a complete error analysis accounting for accumulated errors at each stage of algorithms -- due to finite precision 
or finite termination of iterative procedures -- together with statements about how close one can get to the 
{\em solution} to a given optimization problem,  as opposed to its {\em optimal value}, the latter having 
in general no necessary connection to the former. 
This is a monumental project that has not received as much attention in the literature as studies of 
complexity based upon function values.  
As we argued, the standard approach for handling inexact computation by assuming summable errors 
does not solve the problem, it just 
distributes it over infinitely many iterates.   An alternative to this was suggested in 
\cite[Section 6]{LLM} and applied in \cite{Luke12} which allows a fixed error over all iterations without 
compromising local linear convergence.   More work in this direction would narrow the gap 
between theory and practice. 

\section*{Appendix}
\label{app:DR-ADMM}
{\em Duality of ADMM and the Douglas--Rachford Algorithm.}
Consider the sequence $\paren{b^k,v^k}_{k\in\Nbb}$ of the Douglas--Rachford iteration \ref{e:DRSalg1}, for the case 
$B:=\partial( J^*\circ(-A^{T}))$; $D:=\partial  H^*$.  Recalling the two-step implementation \eqref{a:DRSalg1'}, 
denote $\bar{p}:=b^k-\eta v^k$ and $p':=q^{k+1}$.  Then \eqref{a:DRSalg1'a} is the proximal step $p'=(I+\eta \partial( J^*\circ(-A^{T})) )^{-1}\bar{p}$ on the 
operator $B=\partial( J^*\circ(-A^{T}))$. If $A$ has full column rank, by \cite[Proposition  3.32(iv)]{Eckstein}, this step can be performed by
\begin{eqnarray}
 \label{e:step1}u^{k+1}&=&\mbox{arg}\min_u\{ J(u)+\langle \bar{p}+\eta v^k, Au \rangle+\tfrac{\eta}{2}\|Au-v^k\|^2\};\\
 \label{e:step11}p'&=&\bar{p}+\eta A u^{k+1}.
\end{eqnarray}
Indeed, since $A$ has full rank, $ J(u)+\langle \bar{p}+\eta v^k, Au \rangle+\tfrac{\eta}{2}\|Au-v^k\|^2$ is a proper
strongly convex function of $u$ and has a unique minimizer $u^{k+1}$. From the optimality 
condition for \eqref{e:step1},
  \[0 \in \partial  J(u^{k+1})+A^{T}(\bar{p}+\eta Au^{k+1})=\partial  J(u^{k+1})+A^{T} p'.\]
\ni Hence, $(u^{k+1},-A^{T} p' ) \in \gph \partial  J$ which implies $(-A^{T} p', u^{k+1}) \in \gph \partial  J^*$. This gives
 \begin{eqnarray*}
  \Leftrightarrow  (p', u^{k+1}) &\in& \gph\paren{\partial  J^*\circ (-A^{T})}\\
    \Leftrightarrow  (p', -A u^{k+1}) &\in & \gph \paren{-A\circ \partial  J^*\circ (-A^{T}) }\subseteq \gph \partial \paren{ J^*\circ (-A^{T})}.
 \end{eqnarray*}
 Using \eqref{e:step11},
\begin{eqnarray*} (p', \tfrac{1}{\eta}(\bar{p}-p') &\in & \gph \partial \paren{ J^*\circ (A^{T})}\\
    \Leftrightarrow p'&=&(I+\eta\partial( J^*\circ(A^{T})) )^{-1}\bar{p}.
 \end{eqnarray*}
Substituting $\bar{p}=b^k-\eta v^k$ in \eqref{e:step1}-\eqref{e:step11} yields
 \begin{eqnarray}
 \label{e:step1a}u^{k+1}&=&\mbox{arg}\min_u\{ J(u)+\langle b^k-\eta v^k+\eta v^k, Au \rangle+\tfrac{\eta}{2}\|Au-v^k\|^2\};\\
 \label{e:step1b}q^{k+1}&=&  b^k-\eta v^k+\eta A u^{k+1}.
\end{eqnarray}

Similarly, if we denote  $\bar{p}:=q^{k+1}+\eta v^k(= b^k+\eta A u^{k+1})$ and $p':=b^{k+1}$, \eqref{a:DRSalg1'b} is the proximal step 
$p'=(I+\eta \partial  H^*)^{-1}\bar{p}$ on the 
operator $D=\partial  H^*$ which can be performed via
\begin{eqnarray*}
 v^{k+1}&=&\mbox{arg}\min_v\{ H (v)-\langle \bar{p}-\eta A u^{k+1}, v \rangle+\tfrac{\eta}{2}\|Au^{k+1}-v\|^2\};
\\p'&=&\bar{p}-\eta v^{k+1}.
\end{eqnarray*}
Substituting $\bar{p}=b^k+\eta A u^{k+1}$,
\begin{eqnarray}\label{e:step2}
\label{e:step2a} v^{k+1}&=&\mbox{arg}\min_v\{ H (v)-\langle b^k+\eta A u^{k+1}-\eta A u^{k+1}, v \rangle+\tfrac{\eta}{2}\|Au^{k+1}-v\|^2\};\\
\label{e:step2b} b^{k+1}&=&b^{k}+\eta Au^{k+1}-\eta v^{k+1}.
\end{eqnarray}
 Now, \eqref{e:step1a}-\eqref{e:step1b} and \eqref{e:step2a}-\eqref{e:step2b} together yield
 \begin{eqnarray*}
u^{k+1}&=&\mbox{arg}\min_u\{ J(u)+\langle b^k, Au \rangle+\tfrac{\eta}{2}\|Au-v^k\|^2\};\\
 v^{k+1}&=&\mbox{arg}\min_v\{ H (v)-\langle b^k, v \rangle+\tfrac{\eta}{2}\|Au^{k+1}-v\|^2\};\\
 b^{k+1}&=&b^{k}+\eta( Au^{k+1}- v^{k+1}).
\end{eqnarray*}
This is the ADMM algorithm \eqref{ADMM1}-\eqref{ADMM3} for the primal problem ($\mathcal P_\lambda$). 
\hfill$\Box$
\bigskip

\medskip


\section*{Acknowledgments}
We thank Jennifer Schubert of the Laser-Laboratorium G\"ottingen for providing us with the STED measurements shown in Fig.~\ref{Fig:edata}.
Thanks also to Jalal Fadili for fruitful discussions and helpful comments during the preparation of this work.

 \end{document}